\documentclass[oneside,english]{article}
\usepackage{fullpage}
\usepackage{amsthm}
\usepackage{amsmath,amssymb,amsfonts,nicefrac}
\usepackage{xspace}
\usepackage{color}
\usepackage{url}
\usepackage{hyperref}
\usepackage{bm}
\usepackage{bbm}
\usepackage{times}

\usepackage{enumitem}

\makeatletter
\theoremstyle{plain}
\newtheorem{thm}{\protect\theoremname}[section]
\theoremstyle{definition}
\newtheorem{problem}[thm]{\protect\problemname}
\theoremstyle{definition}
\newtheorem{defn}[thm]{\protect\definitionname}
\theoremstyle{plain}
\newtheorem{lem}[thm]{\protect\lemmaname}
\theoremstyle{plain}
\newtheorem{prop}[thm]{\protect\propositionname}
\theoremstyle{plain}
\newtheorem{cor}[thm]{\protect\corollaryname}
\theoremstyle{remark}
\newtheorem{claim}[thm]{\protect\claimname}

\makeatother

\usepackage{babel}
\providecommand{\claimname}{Claim}
\providecommand{\corollaryname}{Corollary}
\providecommand{\definitionname}{Definition}
\providecommand{\lemmaname}{Lemma}
\providecommand{\problemname}{Problem}
\providecommand{\propositionname}{Proposition}
\providecommand{\theoremname}{Theorem}

\title{On the largest product-free subsets of the alternating groups}

\date{}

\author{Peter Keevash\thanks{Mathematical Institute, University of Oxford, UK.
Supported by ERC Advanced Grant 883810.} \and Noam Lifshitz
\thanks{Einstein Institute of Mathematics, Hebrew University, Jerusalem, Israel.} \and Dor Minzer
\thanks{Department of Mathematics, Massachusetts Institute of Technology, Cambridge, USA.
Supported by a Sloan Research Fellowship.}}

\begin{document}
\global\long\def\matA{\mathrm{A}}%
\global\long\def\matB{\mathrm{B}}%
\global\long\def\matC{\mathrm{C}}%

\global\long\def\Astruc{\mathrm{A}_{\text{struc}}}%

\global\long\def\Bstruc{\mathrm{B}_{\text{struc}}}%
\global\long\def\Cstruc{\mathrm{C}_{\text{struc}}}%

\global\long\def\Asstruc{\mathrm{A}_{\text{struc,\ensuremath{\star}}}}%

\global\long\def\Bsstruc{\mathrm{B}_{\text{struc,\ensuremath{\star}}}}%
\global\long\def\Csstruc{\mathrm{C}_{\text{struc,\ensuremath{\star}}}}%

\global\long\def\Atstruc{\mathrm{A}_{\text{struc,\ensuremath{\star}}}^{t}}%

\global\long\def\Btstruc{\mathrm{B}_{\text{struc},\star}^{t}}%
\global\long\def\Ctstruc{\mathrm{C}_{\text{struc,\ensuremath{\star}}}^{t}}%

\global\long\def\sf#1{\mathsf{#1}}%
\global\long\def\norm#1{\left\Vert #1\right\Vert }%

\global\long\def\Apseudo{\mathrm{A}_{\text{rand}}}%
\global\long\def\Bpseudo{\mathrm{B}_{\text{rand}}}%
\global\long\def\Cpseudo{\mathrm{C}_{\text{rand}}}%
\global\long\def\Expect#1#2{\stackrel[#1]{}{\mathbb{E}}\left[#2\right]}%
\global\long\def\T{\mathrm{T}}%
\global\long\def\inner#1#2{\left\langle #1,#2\right\rangle }%
\global\long\def\eps{\epsilon}%

\global\long\def\comment#1{\text{\fbox{{\tt #1}}}}%

\global\long\def\E{\mathop{\mathbb{E}}}%

\global\long\def\card#1{\left|#1\right|}%
\global\long\def\ee{\asymp}%
\global\long\def\sett#1#2{\left\{  #1:\,#2\right\}  }%
\global\long\def\es{\varnothing}%
\global\long\def\TT{\Theta}%
\global\long\def\sub{\subseteq}%
\global\long\def\sm{\backslash}%

\newcommand{\mb}[1]{\mathbb{#1}}

\newcommand{\aA}{\alpha}
\newcommand{\bB}{\beta}
\newcommand{\gG}{\gamma}
\newcommand{\dD}{\delta}
\newcommand{\iI}{\iota}
\newcommand{\kK}{\kappa}
\newcommand{\zZ}{\zeta}
\newcommand{\lL}{\lambda}
\newcommand{\tT}{\theta}
\newcommand{\sS}{\sigma}
\newcommand{\oO}{\omega}
\newcommand{\GG}{\Gamma}
\newcommand{\DD}{\Delta}
\newcommand{\OO}{\Omega}
\newcommand{\Ss}{\Sigma}
\newcommand{\LL}{\Lambda}
\newcommand{\Ups}{\Upsilon}

\maketitle

\begin{abstract}
A subset $A$ of a group $G$ is called product-free if there is no solution to $a=bc$ with $a,b,c$ all in $A$.
It is easy to see that the largest product-free subset of the symmetric group $S_n$ is obtained by taking
the set of all odd permutations, i.e.\ $S_n \sm A_n$, where $A_n$ is the alternating group.
By contrast, it is a long-standing open problem to find the largest product-free subset of $A_n$.
We solve this problem for large $n$, showing that the maximum size is achieved
by the previously conjectured extremal examples, namely families of the form
$\sett{\pi}{\pi(x)\in I, \pi(I)\cap I=\es}$ and their inverses.
Moreover, we show that the maximum size is only achieved by these extremal examples,
and we have stability: any product-free subset of $A_n$ of nearly maximum size
is structurally close to an extremal example.
Our proof uses a combination of tools from Combinatorics and Non-abelian Fourier Analysis,
including a crucial new ingredient exploiting some recent theory
developed by Filmus, Kindler, Liftshitz and Minzer
for global hypercontractivity on the symmetric group.
\end{abstract}

\section{Introduction}

The long-standing problem of determining the largest product-free
set in the alternating group $A_{n}$ has been recently highlighted
by Ben Green \cite{Green}, who credits Edward Crane for the conjectured
extremal examples, which are families of the form
\[
F_{I}^{x}:=\sett{\pi}{\pi(x)\in I,\pi(I)\cap I=\es}
\]
 and their inverses. Writing $\mu$ for uniform measure on $A_{n}$
and $|I|=t\sqrt{n}$, one can calculate $\mu(F_{I}^{x})\approx te^{-t^{2}}n^{-1/2}$,
which suggests the conjecture that the maximum measure should be $\TT(n^{-1/2})$,
and more precisely that it should be $\sim1/\sqrt{2en}$. Improving
earlier bounds of Kedlaya \cite{kedlaya1997large} and Gowers \cite{Gowers},
the conjecture was proved up to logarithmic factors by Eberhard \cite{Eberhard},
who showed that any product-free $A\sub A_{n}$ has $\mu(A)=O(n^{-1/2}\log^{7/2}n)$.
Our main result here answers the question completely, as follows.
\begin{thm}
\label{thm:main} Suppose $n$ is sufficiently large and $A\sub A_{n}$
is a product-free subset of maximum size. Then $A$ or $A^{-1}$ is
some $F_{I}^{x}$.
\end{thm}

\subsection{99\% stability}

We also obtain the following `99\% stability' result, showing that
any large product-free subset of $A_{n}$ is essentially contained
in an $F_{I}^{x}$ or its inverse. Our stability result holds for
sets whose measure is much smaller than the extremal family.
\begin{thm}
\label{thm:99=000025 Stability} Suppose $n$ is sufficiently large
and $A\sub A_{n}$ is a product-free set with $\mu\left(A\right)\ge n^{-0.66}$.
Then there is some $F_{I}^{x}$ such that $\mu(A\sm F_{I}^{x})<n^{-0.66}$
or $\mu(A^{-1}\sm F_{I}^{x})<n^{-0.66}$.
\end{thm}

Moreover, if $A \sub A_n$ is a product-free set
with size very close to the maximum
we show that $A$ or $A^{-1}$
is contained in some $F_{I}^{x}$.
\begin{thm} \label{thm:main+}
There exists an absolute constant $c$ such that
if $n$ is sufficiently large
and $A\sub A_{n}$ is a product-free set with
$\mu(A)>\max_{I,x}\mu\left(F_{I}^{x}\right)-\frac{c}{n}$
then there is some $F_{I}^{x}$ such that $A\subseteq F_{I}^{x}$
or $A^{-1}\subseteq F_{I}^{x}$.
\end{thm}

\subsection{1\% stability}

We also study the following `1\% stability' problem.
\begin{problem}
Suppose that $A\subseteq A_{n}$ is a product-free set of density
$>1/n^{C}$ for an absolute constant $C$. What can be said about
the structure of $A$?
\end{problem}

\subsection*{Dictators and $t$-umvirates}

Here we describe the structures appearing in the answer to this problem.
A set of the form $\mathcal{D}_{i\to j}=\left\{ \sigma\in A_{n}:\,\sigma(i)=j\right\} $
is called a \emph{dictator}. Let $D_{1},\ldots D_{t}$ be distinct
dictators that have a nonempty intersection. Following Friedgut \cite{friedgut2008measure},
we call their intersection a $t$-\emph{umvirate}. Equivalently, a
$t$-umvirate corresponds to ordered sets $I=(i_{1},\ldots,i_{t}),J=(j_{1},\ldots,j_{t})$
and is given by letting $\mathcal{U}_{I\to J}$ be the set of permutations
that send each $i_{k}$ to $j_{k}$.

\subsection*{Densities}

Let $A,B\subseteq A_{n}.$ The \emph{density} of $A$ inside $B$,
denoted by $\mu_{B}\left(A\right)$ is $\frac{|A\cap B|}{|B|}$. We
may restrict a subset $A\subseteq S_{n}$ to $\mathcal{U}_{I\to J}$
by setting $A_{I\to J}=A\cap\mathcal{U}_{I\to J}.$ We view $A_{I\to J}$
as a set inside the ambient space $\mathcal{U}_{I\to J}$.
We therefore write $\mu\left(A_{I\to J}\right)
= \frac{|A \cap \mathcal{U}_{I\to J}|}{|\mathcal{U}_{I\to J}|}$.

\subsection*{Product-free sets correlate with $t$-umvirates}

Our next theorem shows that any product-free set
that is somewhat dense has some local structure.
This is analogous to the strong local structure exhibited by
the extremal families $F_{I}^{x}$,
which have $\Theta\left(1\right)$ density
inside each dictator $1_{x\to i}$ with $i \in I$,
as when $|I|=\Theta(\sqrt{n})$ a random permutation sends $I$
to its complement with probability $\Theta(1)$.
We show that a similar, albeit weaker, phenomenon holds
for product-free sets with any polynomial density
that can be much smaller than that in the extremal examples:
such sets have a density bump inside a $t$-umvirate.
\begin{thm}
\label{thm:global} Fix $r \in \mb{N}$, suppose $n$ is sufficiently large,
and $A\sub A_{n}$ is product-free with $\mu(A)>n^{-r}$.
Then there exists some $t$-umvirate with $t\le 4r$
in which $A$ has density at least $n^{t/4}\mu(A)$.
\end{thm}

This 1\% stability result will be deduced from a combination of the trace method,
a recent level-$d$ inequality due to Filmus, Kindler, Lifshitz and
Minzer \cite{FKLM}, and novel upper bounds on eigenvalues of Cayley
graphs over the symmetric group.

\subsection{Globalness\label{subsec:global}}

Theorem \ref{thm:global} naturally leads us to define the notion
of `globalness', which plays a crucial role throughout the paper.
This is intuitively a pseudorandomness property stating that membership
is not determined by local information, which one can think of as
being the polar extreme to dictators.

More precisely, we make the following definition, saying that small
restrictions do not have large measure. We include two versions of
the same concept with different parameterisations, as we need the
first version when referring to \cite{FKLM}, but the second version
is more natural for the applications in our paper.
\begin{defn} \label{dfn:globalset}
We say $A\subseteq A_{n}$ is $(t,\eps)$-global if the density of
$A$ inside each $t$-umvirate is $<\eps^{2}$. We say that $A$ is
relatively $(t,K)$-global if the density of $A$ inside each $t$-umvirate
is at most $K\mu(A)$.
\end{defn}

\subsection{Product-free triples}

We also consider the following cross version of the question. Given
$A,B,C\sub A_{n}$, we say that the triple $(A,B,C)$ is product-free
if there is a solution of\footnote{To clarify our notation for composition of permutations: we mean that
$a(b(x))=c(x)$ for all $x\in[n]$.} $ab=c$ with $a\in A$, $b\in B$, $c\in C$. In particular, $(A,A,A)$
is product-free if and only if $A$ is product-free. We consider the
following problem.
\begin{problem}
\label{prob:Cross problem} What are the possible sizes of product
free triples $\left(A,B,C\right).$
\end{problem}

Problem \ref{prob:Cross problem} incorporates the well-studied problem of
upper bounding sizes of independent sets inside Cayley graphs
(see e.g. Ellis, Filmus and Pilpel \cite{ellis2011intersecting}),
which concerns the special case $A=A^{-1}$ and $B=C.$

For simplicity we restrict ourselves to the question of maximising
$\min(\mu(A),\mu(B),\mu(C))$ when $(A,B,C)$ is product-free.
Gowers \cite{Gowers} established an upper bound of $\left(n-1\right)^{-1/3}$
by showing that if $(A,B,C)$ is product-free then $\mu(A)\mu(B)\mu(C)$
$\le\frac{1}{n-1}$. Eberhard \cite{Eberhard} improved this to
$O(n^{-1/2}\log^{3.5}n)$ by showing that one of $\mu(A)\mu(B)$,
$\mu(A)\mu(C)$, $\mu(B)\mu(C)$ is $O(n^{-1}\log^{7}n)$. We obtain
the following bound, which gives the correct order of magnitude, as
shown by examples of the form
\[
A=1_{I\to\overline{I}}:=\sett{\pi}{\pi(I)\cap I=\es}
\qquad
C=B=1_{x\to I}:=\sett{\pi}{\pi(x)\in I}.
\]

\begin{thm}
\label{thm:xmain} Let $n$ be sufficiently large and $A,B,C\sub A_{n}$
with $(A,B,C)$ product-free. Then one of $\mu(A)$, $\mu(B)$, $\mu(C)$
is at most $100\sqrt{\log n}/\sqrt{n}$.
\end{thm}

In fact, we show corresponding stability results
that are a bit harder to state (see Theorem \ref{thm:Cross stability result}).
Roughly speaking, they say that any sufficiently dense product-free triple
has the approximate form $\left(1_{I\to\overline{J}},1_{x\to I},1_{x\to J}\right).$
Moreover, we also prove a version of Theorem \ref{thm:global} that finds
structure in product-free triples that are only polynomially dense.

As discussed later in the introduction, our approach can be viewed
as an non-abelian analogue of Roth's bound for sets of integers with
no three-term arithmetic progression, whereby we improve the earlier
approaches of Gowers and Eberhard by establishing a form of the `Structure
versus Randomness' dichotomy. We achieve this by exploiting some recent
theory developed by Filmus, Kindler, Lifshitz and Minzer \cite{FKLM}
for global hypercontractivity on the symmetric group, and by also developing
some further theory of the `Cayley operators' associated to global subsets.

\subsection{Techniques} \label{sec:tec}

We call a set of the form $1_{x\to I}=\left\{ \sigma:\sigma\left(x\right)\in I\right\} $
a \emph{star}. We call the set $\left\{ \sigma:\sigma^{-1}\left(x\right)\in I\right\} =1_{x\to I}^{-1}$
an \emph{inverse star.}

The main steps in the proof of our main theorem (Theorem \ref{thm:main})
are as follows.
\begin{enumerate}
\item Achieving dictatorial structure: We show that $A$ has large density
inside many dictators. In fact, we show that in some sense, the product
freeness of $A$ is completely explained by its densities inside dictators.
\item Achieving star structure: We then upgrade our dictatorial structure
into a tighter star structure. We find some $S$
that is either a star or an inverse star such that $|A \sm S|$ is small
and $A$ has significant density in each restriction defined by $S$.
\item Bootstrapping: Using the approximate star structure,
we deduce our exact results from further stability analysis
showing that any small deviation from the structure
leads to a suboptimal configuration.
\end{enumerate}

\subsection*{Gowers' approach: the second eigenvalue}

We start out along the path established by Gowers \cite{Gowers}.
His idea was to express the number of products in $A\sub A_{n}$
of density $\alpha$ as the sum of a `main term' and an `error term',
where the error term is smaller in magnitude than the
main term when $\alpha$ is large, so the number of products cannot
be zero. The main term $\alpha^{3}|A_{n}|^{2}$ is the expected number
of products in a random set of density $\mu\left(A\right)$, whereas
his bound for the error term has order $\left(n-1\right)^{-1/2}\alpha^{3/2}|A_{n}|^{2}$,
so when $A$ is product-free he obtained the bound $\alpha\le\left(n-1\right)^{-1/3}$.

We will now outline his argument. Let $f=1_{A}$ be the indicator
function of $A\sub A_{n}$ and $T$ be the linear operator on $L^{2}(A_{n})$
defined by
\[
(Tg)(\pi)=\mathbb{\mathbb{E}}_{\sigma\sim A_{n}}\left[f\left(\sigma\right)g(\sigma\circ\pi)\right].
\]
 Then $A$ contains a product if and only if $\inner f{Tf}>0$. Let
$V'$ be the space of functions of expectation 0. Then we have $\inner f{Tf}=\alpha^{3}+\inner{f'}{Tf'}$,
where $f'=f-\alpha\in V'$. Writing $T^{*}$ for the adjoint of $T$,
some Representation Theory of $S_{n}$ (discussed in more detail below)
tells us that the self-adjoint operator $T^{*}T$ acts on $V'$ with
all eigenvalues bounded absolutely by $\frac{\alpha}{n-1}$. We deduce
\begin{align*}
\inner{f'}{Tf'}^{2} & \le\|f'\|_{2}^{2}\|Tf'\|_{2}^{2}=\|f'\|_{2}^{2}\inner{f'}{T^{*}Tf'}\\
 & \le\|f'\|_{2}^{3}\|T^{*}Tf'\|_{2}\le\|f'\|_{2}^{4}\alpha/\left(n-1\right)\le\alpha^{3}/\left(n-1\right).
\end{align*}
 Thus we can deduce $\inner f{Tf}>0$ if $\alpha^{3}>\sqrt{\alpha^{3}/\left(n-1\right)}$,
i.e. if $\alpha>\left(n-1\right)^{-1/3}$.

\subsection{Beyond spectral gap: the degree decomposition}

This path was continued by Eberhard \cite{Eberhard}, whose improved
bound is based on a refined analysis of the main contribution to the
error term (he credits Ellis and Green for suggesting this approach),
replacing the basic decomposition $f=\alpha+f'$
by a refined `level' decomposition $f=\sum_{d=0}^{n-1}f^{=d}$.

Instead of working with functions on $A_{n}$,
henceforth it will be more convenient
to work with functions on $S_{n}$ that are supported on $A_{n}$,
so we will now let $\alpha = |A|/n!$
denote the density of $A$ inside $S_{n}$.
This allows us to import machinery developed for $S_n$
without reworking it for $A_n$, although it introduces
some slight inconvenience in keeping track
of factors of $2$ in the calculations.

For $d\le n$, let $W_{d}$ be the linear subspace of $L^{2}(S_{n})$
generated by indicators of $d$-umvirates. The degree of a nonzero
function $f$ is the minimal $d$ such that $f\in W_{d}$.
We define $V_{=d}=W_{d}\cap W_{d-1}^{\bot}$
and note that the spaces $V_{=d}$
form an orthogonal decomposition of $L^{2}(S_{n})$,
known as the \emph{degree} decomposition of $S_{n}$.
For a function $f$ on $S_{n}$ we write $f^{=d}$ for the projection
of $f$ onto $V_{=d}$.

To set up Eberhard's refined analysis,
we write $f=2\alpha+2f^{=1}+f''$ with
\[
f'' \in V''=\left\{ g-2\mathbb{E}\left[g\right]-2g^{=1}:\,\text{g is supported on }A_{n}\right\}.
\]
Again, some Representation Theory shows that $T^* T$ acts on $V''$
with all eigenvalues $O(\alpha/n^{2})$.
Similarly to the above argument this implies
\[
\inner f{Tf}=2\alpha^{3}+2\inner{f^{=1}}{Tf^{=1}}+\inner{f''}{Tf''},
\]
where $\inner{f''}{Tf''}=O(\alpha^{3/2}/n)$ is negligible when $n^{-2/3}=o\left(\alpha\right)$.
Thus it suffices to control the `linear term' $\inner{f^{=1}}{Tf^{=1}},$
which turns out to be equal to
\[
\mathbb{E}_{\sigma,\pi\sim S_{n}}\left[f^{=1}\left(\sigma\right)f^{=1}\left(\pi\right)f^{=1}\left(\sigma\pi\right)\right].
\]

However, it is not generally true that the linear term is small compared
with the main term. Indeed, this would imply `product
mixing', i.e. that the number of products is close to that in a random
set of the same density, but Eberhard \cite{Eberhard}
constructed examples with significantly
more products when $\alpha=o\left(n^{-1/3}\right)$.
The key to his approach is that the above counting
argument only needs a lower bound on the linear term, and that this
exhibits much nicer concentration properties than the upper bound.
Nevertheless, even these better estimates break down for densities
within a polylogarithmic factor of the optimum bound.

\subsection*{Our approach: structure versus randomness}

Our approach to stability can be viewed as a further refinement of
this method that is close in spirit to Roth's theorem, in that it
is analogous to the `structure versus randomness' dichotomy: for Roth's
theorem, if the error term is large then $A$ correlates with an arithmetic
progression, whereas in our setting, if the linear term cancels the
main term then $A$ correlates with dictators.

Our starting point for this strategy is the formula
\[
f^{=1}(\pi)=\sum_{i,j\in[n]}a_{ij}x_{i\to j},
\]
 where $x_{i\to j}\left(\pi\right)=1_{\pi\left(i\right)=j},$ and
each
\[
a_{ij}=(1-1/n)(\mu(A_{i\to j})-\mu(A))
\]
 measures the correlation of $f$ with a dictator. A large value of
$a_{i,j}$ corresponds to a large density inside a dictator. On the
other hand, having all $a_{i,j}$ of the same order of magnitude as
$\mu(A)$ can be interpreted as pseudorandomness. Some calculations
reveal that
\[
\|f^{=1}\|_{2}^{2}=\tfrac{1}{n-1}\sum_{i,j\in\left[n\right]}a_{ij}^{2},
\]
 and that
\begin{align*}
\mathbb{E}_{\sigma,\pi\sim S_{n}}\left[f^{=1}\left(\sigma\right)f^{=1}\left(\pi\right)f^{=1}\left(\sigma\pi\right)\right] & =\frac{1}{\left(n-1\right)^{2}}\sum_{i,j,k\in\left[n\right]}a_{ij}a_{jk}a_{ik}.
\end{align*}
 These formulae alone do not suggest any structural properties for
a contribution of $-\alpha^{3}$ to the left hand side;
in principle, $\TT(n/\alpha)$ values of $a_{i,j}=-\alpha$ could
contribute $\tfrac{1}{n^{2}}\TT(n/\alpha)^{3/2}(-\alpha)^{3}$, so
when $\alpha=o\left(n^{-1/3}\right)$ we seem to have no structure.

Such extreme situations can be ruled out by concentration of measure,
which is a key tool in Eberhard's approach, but seems doomed to give
up logarithmic factors. However, much stronger structure can be extracted
from a recent hypercontractive inequality of Filmus, Kindler, Lifshitz,
and Minzer \cite{FKLM}.

It will be convenient for us to put the coefficients $a_{ij}$ inside
a matrix $\matA=\left(a_{ij}\right)$ and to equip real valued $n\times n$
matrices with the inner product $\left\langle \left(m_{ij}\right),\left(n_{ij}\right)\right\rangle =\sum_{i=1}^{n}\sum_{j=1}^{n}m_{ij}n_{ij}.$
We then have
\[
\|f^{=1}\|_{2}^{2}=\frac{1}{n-1}\|\matA\|_{2}^{2}
\]
 and
\[
\mathbb{E}_{\sigma,\pi}\left[f^{=1}\left(\pi\right)f^{=1}\left(\sigma\right)f^{=1}\left(\sigma\pi\right)\right]
=\frac{1}{(n-1)^2}\left\langle \matA^{2},\matA\right\rangle .
\]
Our idea is to decompose $\matA=\Apseudo+\Astruc+\matA_{-}$ as follows.
We let the matrix $\text{\ensuremath{\matA}}_{-}$
consist of the negative coefficients $a_{ij}$, where the other coefficients
are replaced by 0. We then set the matrix $\Astruc$ to consist of
the `large' values $a_{ij}\ge\epsilon$, for some carefully chosen $\epsilon>0$.
Finally, we let $\Apseudo$
consist of the `small' positive coefficients $a_{ij}\in\left(0,\epsilon\right).$

We then expose the dictatorial structure of $A$ by showing that
\begin{align}
\mathbb{E}_{\sigma,\pi}\left[f^{=1}\left(\pi\right)f^{=1}\left(\sigma\right)f^{=1}\left(\sigma\pi\right)\right] & \ge\left\langle \Astruc\matA_{-},\Astruc\right\rangle +\left\langle \matA_{-}\Astruc,\Astruc\right\rangle \nonumber \\
 & \qquad
 +\left\langle \Astruc^{2},\matA_{-}\right\rangle +o\left(\alpha^{3}\right).\label{eq: dictatorial structure}
\end{align}
 This could be understood intuitively as follows. Writing $AB=\left\{ \sigma\tau:\sigma\in A,\tau\in B\right\} $,
we see that the dictators satisfy $\mathcal{D}_{j\to k}\cdot\mathcal{D}_{i\to j}=\mathcal{D}_{i\to k}.$
When expanding (\ref{eq: dictatorial structure}) in terms of the
coefficients $a_{ij}$, it shows that the only significant negative
contribution to
\[
\mathbb{E}_{\sigma,\pi}\left[f^{=1}\left(\pi\right)f^{=1}\left(\sigma\right)f^{=1}\left(\sigma\pi\right)\right]
\]
 comes from triples $a_{jk}a_{ij}a_{ik}$ corresponding to dictators
$\mathcal{D}_{j\to k}\cdot\mathcal{D}_{i\to j}=\mathcal{D}_{i\to k}$,
with the property that $A$ has a large density in two of the dictators
and a small density in the remaining one.

To establish (\ref{eq: dictatorial structure}) we will expand the
inner product $\left\langle \matA^{2},\matA\right\rangle $ in terms
of the matrices $\matA_{-},\Astruc,$ $\Apseudo$ and use the inequality
\[
\left\langle MN,S\right\rangle \le\|M\|_{2}\|N\|_{2}\|S\|_{2}
\]
to upper bound the undesirable terms. Our proof will thus crucially rely
on upper bounds for $\|\Apseudo\|_{2}$ and $\|\matA_{-}\|_{2}$,
which we will establish via a hypercontractive inequality for global functions,
as discussed in the next subsection.

\subsection*{Level-1 inequalities}

The relationship between hypercontractive inequalities and inequalities
that upper bound the `level-$1$ weight' $\|f^{=1}\|_{2}^{2}$ is well-known in
the context of Boolean functions $f\colon\left\{ 0,1\right\} ^{n}\to\left\{ 0,1\right\} .$
There an inequality of the form
\[
\|f^{=1}\|_{2}^{2}\le2\mathbb{E}^{2}\left[f\right]\log\left(\frac{1}{\mathbb{E}\left[f\right]}\right)
\]
 is true for all Boolean functions and is known as the \emph{level-$1$ inequality}
(see O'Donnell \cite[Chapter 5]{o2014analysis}). As
\[
\frac{1}{\left(n-1\right)^{2}}\left\langle \matA^{2},\matA\right\rangle \le\frac{\|\matA\|_{2}^{3}}{\left(n-1\right)^{2}}=\frac{\|f^{=1}\|_{2}^{3}}{\sqrt{n-1}},
\]
a similar level-$1$ inequality for the symmetric group would be most
desirable, as it would imply that
$\frac{\|f^{=1}\|_{2}^{3}}{\sqrt{n-1}}$ is negligible compared to $\alpha^{3}$.
However, such an inequality is not true in general, as it fails for dictators,
and more generally for $t$-umvirates with small $t$.

The local nature of the obstructions suggests that the approach
could be rescued by proving a level-$1$ inequality for global functions.
Indeed, this was achieved in the analogous setting of general product spaces
by Keevash, Lifshitz, Long and Minzer \cite{KLLM},
who developed a hypercontractivity theory for global functions
that has recently been a fruitful source of many applications.
The corresponding results in the setting of the symmetric group
have recently been established by Filmus, Kindler, Lifshitz and Minzer \cite{FKLM}.
Their level-$1$ inequality shows that if $f\colon S_{n}\to\left\{ 0,1\right\} $
is $\left(2,\epsilon\right)$-global then
\[
\|f^{=1}\|_{2}^{2}\le2\epsilon\mathbb{E}\left[f\right]\log^{O\left(1\right)}\left(\frac{1}{\epsilon}\right).
\]

This inequality cannot be applied directly to our setting,
as we cannot guarantee that $f$ has small density inside each 2-umvirate.
However, we are able to extend their approach to obtain the upper bounds
\[
\|\matA_{-}\|_{2}\le\mathbb{E}\left[f\right]\log^{O\left(1\right)}\left(2/\mathbb{E}\left[f\right]\right)
\]
 and
\[
\|\Apseudo\|_{2}^{2}\le\epsilon\mathbb{E}\left[f\right]\log^{O\left(1\right)}\left(\frac{2}{\epsilon}\right).
\]
 The key idea is to apply the hypercontractive result of \cite{FKLM}
to
\[
f_{-}:=\sum_{a_{ij}<0}a_{ij}\left(x_{i\to j}-\frac{1}{n}\right)
\]
 and to
\[
f_{\mathrm{rand}}:=\sum_{0<a_{ij}<\epsilon}a_{ij}\left(x_{i\to j}-\frac{1}{n}\right).
\]

These inequalities will establish (\ref{eq: dictatorial structure}).
As an expository simple case of our argument, in the next section
we will show that this can be used to reprove Eberhard's result,
by setting $\epsilon=1$, so that $\Astruc=0$.
Extracting the star structure for smaller $\epsilon$ is considerably more complicated,
so we defer an overview of this part of the argument to subsection \ref{sec:starover}.

\subsection{From star structure to extremal families} \label{subsec:restrict}

Once we know that $A$ is almost contained in a star or inverse star $S$, say $S=1_{x\to I}$,
then it is not too hard to show that it is in fact almost contained in $F_{I}^{x}$.
In other words, we wish to show for each $i,i'\in I$ that $A$ has small density
inside the dictator $\mathcal{D}_{i\to i'}$.
We accomplish this by inspecting the triple $(A_{i\to i'},A_{x\to i},A_{x\to i'})$.
Such triples should be intuitively considered as product-free
after factoring out the corresponding dictators with
\[
\mathcal{D}_{i\to i'}\mathcal{D}_{x\to i}=\mathcal{D}_{x\to i'}.
\]
To formalise this, we would like a variant of Eberhard's bound
that holds not only for product-free triples $A',B',C'\subseteq A_{n},$
but also for product-free sets $A',B',C'$ living inside compatible dictators.
This can be achieved by the following transformation.
We set $A'=\left(i'\,n\right)A\left(n\,i\right)$,
$B'=\left(i\,n\right)A\left(n\,x\right)$, and $C'=\left(i'\,n\right)A(x\,n).$
The transformation from $\left(A,B,C\right)$ to $\left(A',B',C'\right)$
preserves products. Moreover, the restrictions
\[
\left(A'_{n\to n},B'_{n\to n},C'_{n\to n}\right)
\]
correspond to the original restrictions
\[
\left(A_{i\to i'},B_{x\to i},C_{x\to i'}\right).
\]
We may then view the triple $\left(A'_{n\to n},B'_{n\to n},C'_{n\to n}\right)$
as subsets of $A_{n-1}$ and translate Eberhard's bounds inside $A_{n-1}$
to the densities of $A_{i\to i'}$, $A_{x\to i}$ and $A_{x\to i'}$.

Similar considerations combined with some more involved structural arguments
will show that $\mu\left(A\setminus F_{I}^{x}\right)$
is much smaller than $\mu\left(F_{I}^{x}\setminus A\right)$,
thereby showing that $F_{I}^{x}$ is extremal.

\subsection{Eigenvalues of global Cayley graphs}

The main tool for proving our 1\% stability result (Theorem \ref{thm:global})
is an upper bound on the eigenvalues of Cayley graphs $\mathrm{Cay}(S_{n},A)$
that correspond to global sets $A$. We believe that the result has independent interest
and will be beneficial in various other applications;
indeed, the study of Cayley graphs on $S_{n}$ and their eigenvalues
is the basis of an entire field of research
(see e.g. Diaconis \cite{diaconis1988group}).

Let $A=A^{-1}$ and let $T_{A}$ be the operator corresponding to
the random walk on $\mathrm{Cay}(G,A)$, given by $T_{A}f(x):=\mathbb{E}_{a\sim A}[f(ax)]$.
As mentioned above, the operator $T_{A}$ preserves each of the spaces $V_{=d}$.
We define
\[
r_{d}(T_A)=\sup_{f\in V_{=d}}\frac{\|\mathrm{T}_A f\|_{2}}{\|f\|_{2}},
\]
or equivalently $r_{d}\left(T_A\right)^2$ is the largest eigenvalue
of the self-adjoint operator $T_A^* T_{A}$ acting on $V_{=d}$.

A useful fact from representation theory gives lower bounds
on the dimension of each eigenspace of $T_{A}$ using the observation
that each is invariant under the action of $S_{n}$ from the right.
Indeed, this was a crucial ingredient in the seminal work of
Ellis, Filmus and Pilpel \cite{ellis2011intersecting},
who derived bounds on the dimensions from
well-known properties of invariant spaces;
in particular, the dimensions are $\TT(n^d)$ when $d=O(1)$.

The trace method provides a fundamental way to exploit this lower bound.
Specifically, writing $m_{\lambda}$ for the multiplicity of an eigenvalue $\lambda$ of
$T_{A}$, we have
\[
\frac{1}{\mu(A)}=\mathrm{tr}(T_{A}^{*}T_{A})
=\sum_{\lambda\in\mathrm{spec}\left(T_{A}\right)}m_{\lambda}\lambda^{2}.
\]
When $A$ is dense (i.e.\ $\mu(A)=\Theta(1)$) we thus obtain
$|\lambda|=O(1/\sqrt{m_{\lambda}})$,
and so $r_{d}(T_A)=O_d(1/n^{d/2})$.
Furthermore, if $A$ is closed under conjugation
then the operator $\mathrm{T}_{A}$
commutes with the action of $S_{n}$ from both sides, in which case
we have $m_{\lambda}\ge \OO(n^d)$,
giving the stronger bound $r_{d}(T_A)=O_{d}(1/n^{d})$.
We show that similar statements hold for global sets,
even when they are quite sparse.
\begin{thm}
\label{thm:radius} Let $C,d>0$ and suppose that $n$ is sufficiently
large. Suppose that $A\subseteq S_{n}$ is relatively $(2d,K)$-global
with $\mu\left(A\right)\ge n^{-C}$. Then
\[
r_{d}(T_A)\le\frac{\sqrt{K}\log^{O\left(d\right)}n}{n^{d/2}}.
\]
Furthermore, if  $A$ is also closed under conjugation then
\[
r_{d}\left(T_A\right)\le\frac{\sqrt{K}\log^{O\left(d\right)}(n)}{n^{d}}.
\]
\end{thm}

Theorem \ref{thm:radius} is the crucial new ingredient for proving
our 1\% stability result for product-free sets of density $\ge n^{-C}$
(Theorem \ref{thm:global}).

\subsection{Organisation of the paper}

We start in the next section with some technical preliminaries
on the representation theory of $S_n$ and also the proofs
of the results discussed immediately above,
i.e.\ new eigenvalue estimates for global sets
and their application to the proof of our 1\% stability result.
Section \ref{sec:lin} considers the analysis of linear functions
on the symmetric group. The main result of this section
will be a level-$1$ inequality for the pseudorandom part of a function,
which in itself will suffice to reprove Eberhard's result.
We then move into more refined arguments that extract structural properties
of product-free sets,
first exposing the dictatorial structure in Section \ref{sec:dictate}
and then the precise star structure in Section \ref{sec:star}.
In Section \ref{sec:boot} we implement our bootstrapping arguments
that refine the approximate structure to deduce our main results,
giving the exact extremal result and strong stability results
for product-free sets in $A_n$. The final section contains some concluding remarks.

\section{1\% stability} \label{sec:1pc}

This section contains some background on the representation theory of $S_n$,
our new eigenvalue estimates for global sets
and their application to the proof of our 1\% stability result.

\subsection{Notation}

We write $X=O(Y)$ to say that there exists an absolute constant $C>0$
such that $X\leq C\cdot Y$, and similarly $X=\Omega(Y)$ to
say that there exists an absolute constant $c>0$ such that $X\geq c\cdot Y$.
We write $X=\Theta(Y)$ to say that $X=O\left(Y\right)$ and also
$X=\Omega\left(Y\right)$. We also write $X\le Y\log^{O\left(1\right)}n$
to say that there exists an absolute constant $C,$ such that $X\le Y\log^{C}n.$

We discuss the space of real-valued functions on $S_{n}$ equipped
with expectation inner product and $L^{p}$-norms. We write
\[
\mu_{A}\left(B\right)=\frac{\left|A\cap B\right|}{\left|A\right|}=\Pr_{a\sim A}\left[a\in B\right].
\]
 We write $\sigma\sim S_{n}$ to denote that $\sigma$ is uniformly
distributed inside $S_{n}$ and $\sigma,\tau\sim S_{n}$ to denote
that they are chosen independently out of $S_{n}$. For a function
$f\colon S_{n}\to\mathbb{R}$ we often write $\mathbb{E}\left[f\right]$
to denote $\mathbb{E}_{\sigma\sim S_{n}}\left[f\left(\sigma\right)\right]$,
although we caution the reader that this usage will depend on context,
as when we define $\mathbb{E}\left[f_{I\to J}\right]$ below
the expectation will be conditioned on the given restriction.

\subsection{Restrictions}

We define restrictions of functions in a manner that naturally generalizes
the notion of restrictions of subsets of $S_{n}$ used in the introduction.
\begin{defn}
For $t\le n$ and $I=\left(i_{1},\ldots,i_{t}\right),J=\left(j_{1},\ldots,j_{t}\right)\subseteq[n]$
ordered sets of size $t$. We denote by $\mathcal{U}_{I\rightarrow J}$
the $t$-umvirate of permutations sending each $i_{l}$ to $j_{l}$.

We define $x_{i\rightarrow j}\colon S_{n}\to\{0,1\}$ by
\[
x_{i\rightarrow j}(\pi)=1_{\pi\left(i\right)=j}.
\]
 Let $I\subseteq[n]$, $J\subseteq[n]$ be ordered sets of some size
$t$. We denote by $f_{I\rightarrow J}\colon\mathcal{U}_{I\rightarrow J}\to\mathbb{R}$
the restriction of $f$ to $\mathcal{U}_{I\rightarrow J}$. We write
\[
\mathbb{E}\left[f_{I\to J}\right]=\mathbb{E}_{\sigma\sim S_{n}}\left[f\left(\sigma\right)|\sigma\in\mathcal{U}_{I\to J}\right]
\]
 and
\[
\|f_{I\to J}\|_{2}^{2}=\mathbb{E}_{\sigma\sim S_{n}}\left[f^{2}\left(\sigma\right)|\sigma\in\mathcal{U}_{I\to J}\right].
\]
 We may identify $\mathcal{U}_{I\to J}$ with $S_{n-t}$ by choosing
a permutation $\sigma$ sending $n-t+l$ to $i_{l}$ for each $l\in\left[t\right]$
and $\pi$ sending $j_{l}$ to $n-t+l$. Then $\pi\mathcal{U}_{I\to J}\sS$
is the set of permutations on $S_{n}$ fixing $\left\{ n-t+1,\ldots,n\right\} $,
which can be identified with $S_{n-t}.$
We will use this identification to import results on the
symmetric group $S_{n-t}$ to the $t$-umvirate $\mathcal{U}_{I\to J}.$
\end{defn}

\subsection{Orthogonal decompositions}

Our proof will use spectral analysis over $S_{n}$,
so we need to recall its level decomposition
and the more refined representation theoretical decomposition
into isotypical components.

\subsection*{The level decomposition}

We start by decomposing according to degree,
as discussed in the introduction.

\begin{defn}
\label{def:level_space} The space $W_{d}$ is the linear span of
the indicators of the $d$-umvirates. We say that a real-valued function
on $S_{n}$ has \emph{degree} at most $d$ if it belongs to $W_{d}$.

By construction, $W_{d-1}\subseteq W_{d}$ for all $d\geq1$. We define
the space of functions of pure degree $d$ as $V_{=d}=W_{d}\cap W_{d-1}^{\perp}$.
It is easy to see that $V_{n}=V_{n-1}$, and so we can decompose each
real-valued function $f\colon S_{n}\to\mathbb{R}$ as $f=f^{=0}+f^{=1}+\ldots+f^{=n-1}$,
where $f^{=i}\in V_{=i}$. We often refer to this decomposition as
the \emph{level decomposition} of $f$.
\end{defn}

Parts of our argument require the following finer decomposition
that further decomposes $f^{=i}$ into more structured pieces.

\subsection*{The representation theoretic decomposition}

Here we will list the properties we require of
the finer decomposition of functions on $S_{n}$
into isotypical components; these can
be found e.g. in \cite[Section 7.2]{FKLM}.
We adopt the following standard notation.
A partition $\lambda$ uniquely corresponds to a Young diagram.
Its transpose $\lambda^{t}$ is obtained by swapping the roles
of the rows and the columns of the Young diagram. We write $\lambda\vdash n$
to denote that $\lambda$ is a partition of $n$. We let $S_{n}$
act on $L^{2}\left(S_{n}\right)$ from the left and from the right
by setting $\tau g=g\left(\tau\pi\right)$ and $g\tau\left(\pi\right)=g\left(\pi\tau^{-1}\right).$
\begin{lem}
\label{lem:Facts from representation theory} There exists an orthogonal
decomposition $L^{2}\left(S_{n}\right)=\bigoplus_{\lambda\vdash n}V_{=\lambda}$
with the following properties:
\begin{enumerate}
\item For some numbers $\dim\left(\lambda\right)$,
each $V_{=\lambda}$ is the direct sum of $\dim\left(\lambda\right)$
irreducible representations of dimension $\dim\left(\lambda\right)$.
\item If $\T\colon L^{2}\left(S_{n}\right)\to L^{2}\left(S_{n}\right)$
commutes with the action of $S_{n}$ either from the left or from
the right then $\mathrm{T}V_{=\lambda}\subseteq V_{=\lambda}$, and so
$\T V_{=d}\subseteq V_{=d}$. (See \cite[Claim 7.4]{FKLM}.)
\item  If $\mathrm{T}$ is self adjoint and commutes with the action from one
side then the dimension of each eigenspace of $\mathrm{T}$ inside
$V_{=\lambda}$ is at least $\dim\left(\lambda\right)$. If it
commutes with the action of $S_{n}$ from both sides then $V_{=\lambda}$
is contained in an eigenspace of $\mathrm{T}$.
(See \cite[Claim 7.5]{FKLM}.)
\item $V_{=\lambda}\le V_{=d}$ if and only if the first row of $\lambda$
is of length $n-d$. In which case we write $d_{\lambda}=d$ and $\tilde{d}_{\lambda}=\min\left(d_{\lambda},d_{\lambda^{t}}\right).$
\item Multiplication by the sign character sends $V_{\lambda}$ to $V_{\lambda^{t}}$.
(See \cite[Lemma 7.3]{FKLM}.)
\item If $n$ is sufficiently large, $d<n/10$ and $\tilde{d}_{\lambda}>d$
then $\dim\left(\lambda\right)>\left(\frac{n}{ed}\right)^{d}$.
(See \cite[Lemma 7.7]{FKLM}.)
\end{enumerate}
\end{lem}

We write $f^{=\lambda}$ for the projection of $f$ onto $V_{\lambda}.$
We identify functions $f\colon A_{n}\to\left\{ 0,1\right\} $ with
function on $S_{n}$ whose value is 0 on the odd permutations. Such
functions satisfy $\mathsf{sign}\cdot f=f$. This gives rise to two
decompositions of $f$ as a sum of elements in $V_{\lambda}$, which
therefore must be equal. The first is $f=\sum_{\lambda\vdash n}f^{=\lambda}$
and the second is $\sum_{\lambda\vdash n}\sf{sign}\cdot f^{=\lambda^{t}}$.
This shows that $\mathsf{sign}\cdot f^{=\lambda}=f^{=\lambda^{t}}.$

\subsection{Operators from functions}

We write $\mathbb{E}\left[f\right]$ for $\mathbb{E}_{\pi\sim S_{n}}\left[f\right]$.
Thus if $f=1_A$ for $A \sub A_n$ then $\mb{E}[f]=|A|/n!$
denotes the density of $A$ in $S_n$, not $A_n$.

For $f\in L^{2}\left(S_{n}\right)$ we define operators
$\mathrm{\mathrm{L}}_{f}$ and $\mathrm{R}_{f}$ on $L^{2}\left(S_{n}\right)$ by
\[
\mathrm{L}_{f}g\left(\sigma\right)=\mathbb{E}_{\pi}\left[f\left(\pi\right)g(\pi\sigma)\right]
\quad \text{ and } \quad
\mathrm{R}_{f}g\left(\sigma\right)=\mathbb{E}_{\pi}\left[g(\sigma\pi)f\left(\pi\right)\right].
\]
When $f=\frac{1_{A}}{\mu\left(A\right)},$ the operator $\mathrm{L}_{f}$
is the operator corresponding to the random walk sending $\mathrm{\sigma}$
to $a\sigma$ for a random $a\in A$. Similarly, $\mathrm{R}_{f}$
corresponds to the random walk that sends $\sigma$ to $\sigma a$.
The operators $\mathrm{L}_{f}$ and $\mathrm{R}_{f}$ commute with
the actions of $S_{n}$ from one side. Indeed,
for any $g \in L^{2}\left(S_{n}\right)$
and $\sigma, \tau \in S_n$ we have
\[
\left(\mathrm{L}_{f}\left(g\right)\tau\right)\left(\sigma\right)=\mathrm{L}_{f}g\left(\sigma\tau^{-1}\right)=\mathbb{E}_{\pi}\left[f\left(\pi\right)g(\pi\sigma\tau^{-1})\right]=\mathrm{L}_{f}\left(g\tau\right)(\sigma).
\]
 Similarly, we have $\mathrm{R}_{h}\left(\tau g\right)=\tau\left(\mathrm{R}_{h}g\right).$
When $f$ is a class function (meaning that $f\left(\sigma\right)$
depends only on the conjugacy class of $\sigma$)
we have $\mathrm{L}_{f}=\mathrm{R}_{f}$.
Indeed, when $f$ is a class function we have
\[
\mathrm{L}_{f}g(\sigma)=\mathbb{E}_{\pi}\left[f\left(\pi\right)g\left(\sigma\pi\right)\right]=\mathbb{E}_{\pi}\left[f\left(\sigma^{-1}\pi\sigma\right)g\left(\pi\sigma\right)\right]=\mathbb{E}_{\pi}\left[f\left(\pi\right)g\left(\sigma\pi\right)\right]=R_{f}g(\sigma).
\]
As mentioned, the trace method plays a crucial role in our work, and
computing the trace of the operators $\mathrm{L}_{f}^{*}\mathrm{L}_{f}$
will allow us to upper bound its eigenvalues.
\begin{lem}
\label{lem:computing trace}
The operators $\mathrm{L}_{f}^{*}\mathrm{L}_{f}$
and $\mathrm{R}_{f}^{*}\mathrm{R}_{f}$
both have trace $\|f\|_{2}^{2}$.
\end{lem}

\begin{proof}
We only consider $\T=\mathrm{L}_{f}$,
as the proof for $\mathrm{R}_{f}$ is similar.
We have
\begin{align*}
\mathrm{tr}\left(\mathrm{T}^{*}\mathrm{T}\right) & =\sum_{\pi\in S_{n}}n!\left\langle \mathrm{T}^{*}\mathrm{T}1_{\pi},1_{\pi}\right\rangle \\
 & =\sum_{\pi\in S_{n}}n!\left\langle \mathrm{T}1_{\pi},\mathrm{T}1_{\pi}\right\rangle
  =\sum_{\pi\in S_{n}}\sum_{\sigma\in S_{n}} \mathrm{T}1_{\pi}(\sigma)^2 \\
 & =\sum_{\pi\in S_{n}}\sum_{\sigma\in S_{n}}\mathbb{E}_{\tau\sim S_{n}}^{2}\left[f\left(\tau\right)1_{\pi}\left(\tau\sigma\right)\right]\\
 & =\left(\frac{1}{n!}\right)^{2}\sum_{\pi\in S_{n}}\sum_{\sigma\in S_{n}}f^{2}\left(\pi^{-1}\sigma\right)
  =\|f\|_{2}^{2}.
\end{align*}
\end{proof}

\subsection{Hypercontractivity of global functions}

In this subsection we state two inequalities from \cite{FKLM}.
To do so, we need the following natural extension
of the definition of globalness from sets to functions.

\begin{defn}
We say $f\colon S_{n}\to\mathbb{R}$ is $(d,\eps)$-global
if for every ordered $I,J\subseteq[n]$ of size $d$
we have $\norm{f_{I\rightarrow J}}_{2}\leq\eps$.
We say $f$ is relatively $(d,K)$-global
if $\norm{f_{I\rightarrow J}}_{2}^2 \leq K\|f\|_2^2$ for each such $I,J$.
\end{defn}

The hypercontractivity inequality of \cite{FKLM} takes the following form.

\begin{thm}
\label{thm:hyp} There exists an absolute constant $C$ such that
the following holds. Let $d,q\in\mathbb{N}$ with $q\geq2$ and
$n\geq q^{C\cdot d^{2}}$. If $f\colon S_{n}\to\mathbb{R}$ is $(2d,\eps)$-global then
\[
\norm f_{q}\leq q^{O(d^2)}\eps^{1-\frac{2}{q}}\norm f_{2}^{\frac{2}{q}}.
\]
\end{thm}

We also require the level-$d$ inequality of \cite{FKLM},
which is a consequence (not immediate) of their hypercontractive inequality,
showing that global functions have low weight on the first $d$ levels.
\begin{thm}
\label{thm:Level d inequality }
There exists $C>0$ such that
if $n\ge2^{Cd^{3}}\log^{Cd}\left(\frac{1}{\|f\|_{2}}\right)$
and $f\colon S_{n}\to\mathbb{Z}$ is $\left(2d,\epsilon\right)$-global
then
\[
\|f^{\le d}\|_{2}\le 2^{Cd^{4}}\|f\|_{2}\epsilon\log^{Cd}\left(\frac{1}{\|f\|_{2}^{2}}\right).
\]
\end{thm}

\subsection{Eigenvalues of global Cayley graphs}

We show the following stronger version of Theorem \ref{thm:radius}.
Let $\mathrm{T}\colon L^{2}\left(S_{n}\right)\to L^{2}\left(S_{n}\right)$
be an operator that commutes with the action of $S_{n}$ either from
the left or from the right. Then we write
\[
r_{d}\left(\mathrm{T}\right)=\mathrm{sup}_{f\in V_{d}}\frac{\|\mathrm{T}f\|_{2}}{\|f\|_{2}}.
\]
 When $\mathrm{T}$ is self-adjoint $r_{d}\left(\mathrm{T}\right)$
is the largest eigenvalue of $\mathrm{T}$ inside $V_{d}$;
in general, $r_{d}\left(\mathrm{T}\right)$ is the square root of the
largest eigenvalue of $\mathrm{T}^{*}\mathrm{T}$.
Theorem \ref{thm:radius} is immediate from the following result
applied with $f=\frac{1_{A}}{\mathbb{E}\left[1_{A}\right]}$,
as $\|f\|_2 = \mu(A)^{-1/2}$, $\epsilon = \sqrt{K\mu(A)}$,
and $\mathrm{L}_{f} = T_A$ is the random walk operator corresponding to $A$.

\begin{thm}
\label{thm:Actual radius}
There is an absolute constant $C$ such that
if $f\colon S_{n}\to\mathbb{Z}$ is $\left(2d,\epsilon\right)$-global
for some $\eps>0$ and $n \ge 2^{Cd^{3}}\log^{Cd}\left(\frac{1}{\|f\|_{2}}\right)$
then
\[
r_{d}\left(\mathrm{L}_{f}\right)\le\frac{2^{Cd^{4}}\|f\|_{2}\epsilon\log^{Cd}\left(\frac{1}{\|f\|_{2}^{2}}\right)}{n^{d/2}}.
\]
If moreover $f$ is a class function then
\[
r_{d}\left(\mathrm{L}_{f}\right)\le\frac{2^{Cd^{4}}\|f\|_{2}\epsilon\log^{Cd}\left(\frac{1}{\|f\|_{2}^{2}}\right)}{n^{d}}.
\]
\end{thm}

To prove Theorem \ref{thm:Actual radius} we rely on the following lemma.
\begin{lem}
\label{lem:rd depends on f=00003Dd}
For each $d$ and $\lL$, the operator $\mathrm{L}_{f}$
agrees with $\mathrm{L}_{f^{=d}}$ on $V_{=d}$
and with $\mathrm{L}_{f^{=\lL}}$ on $V_{=\lL}$.
\end{lem}

\begin{proof}
Let $R_{\sigma}$ be the operator sending $g$ to $g\sigma.$ Then
$R_{\sigma}$ commutes with the action of $S_{n}$ from the left.
By Lemma \ref{lem:Facts from representation theory} it therefore
preserves the spaces $V_{=d}$ and $V_{=\lambda}.$ Let $g\in V_{=d}$
and let $\sigma\in S_{n}$. Then as $R_{\sigma}g$ is in $V_{=d}$
and as $f^{=d}$ is the projection of $f$ onto $V_{=d}$ we have
\[
L_{f}\left(g\right)\left(\sigma\right)=\left\langle f,R_{\sigma}g\right\rangle =\left\langle f^{=d},R_{\sigma}g\right\rangle =\mathrm{L}_{f^{=d}}g\left(\sigma\right).
\]
 The proof that $\mathrm{L}_{f}$ agrees with $\mathrm{L}_{f^{=\lambda}}$
on $V_{=\lambda}$ is similar.
\end{proof}
\begin{proof}[Proof of Theorem \ref{thm:Actual radius}]
 Let $f\colon S_{n}\to\mathbb{Z}$ be $\left(2d,\epsilon\right)$-global.

The trace of the operator $\mathrm{L}_{f^{=d}}^{*}\mathrm{L}_{f^{=d}}$
is $\|f^{=d}\|_{2}^{2}$ by Lemma \ref{lem:computing trace}. By Lemma
\ref{lem:rd depends on f=00003Dd}, and standard linear algebra we
have
\[
r_{d}\left(\mathrm{L}_{f}\right)=r_{d}\left(\mathrm{L}_{f^{=d}}\right)=\sqrt{r_{d}\left(\mathrm{L}_{f^{=d}}^{*}\mathrm{L}_{f^{=d}}\right)}.
\]
 On the other hand, the trace of a self-adjoint operator is the sum of its eigenvalues.
 Applying Lemma \ref{lem:Facts from representation theory} gives
\[
\min_{\lambda:\,d_{\lambda}=d}\dim\left(\lambda\right)\cdot r_{d}\left(\mathrm{L}_{f^{=d}}^{*}\mathrm{L}_{f^{=d}}\right)\le\mathrm{tr}\left(\mathrm{L}_{f^{=d}}^{*}\mathrm{L}_{f^{=d}}\right)=\|f^{=d}\|_{2}^{2}.
\]
 Putting everything together, plugging in Theorem \ref{thm:Level d inequality }
and Lemma \ref{lem:Facts from representation theory} we obtain
\[
r_{d}\left(L_{f}\right)\le\frac{\|f^{=d}\|_{2}}{\sqrt{\min_{\lambda:\,d_{\lambda}=d}\dim\left(\lambda\right)}}\le\frac{\left(ed\right)^{d/2}2^{Cd^{4}}\|f\|_{2}\epsilon\log^{Cd}\left(\frac{1}{\|f\|_{2}}\right)}{n^{d/2}},
\]
 for some absolute constant $C.$ This implies that the theorem holds
with $2C$ replacing $C$. When $f$ is a class function the same
proof works with $\dim\left(\lambda\right)$ replaced by $\dim\left(\lambda\right)^{2}$
to give
\[
r_{d}\left(\mathrm{L}_{f}\right)\le\frac{2^{2Cd^{4}}\|f\|_{2}\epsilon\log^{Cd}\left(\frac{1}{\|f\|_{2}}\right)}{n^{d}}.
\]
\end{proof}

\subsection{The trace bound in high dimensions}

The above upper bound on $r_{d}\left(\mathrm{L}_{f}\right)$ is complemented
by the following simpler bound that
is more effective when $\tilde{d}_{\lambda}$ is large.
\begin{lem}
\label{lem:trace method  } Let $d<\frac{n}{10}$ and suppose that
$\tilde{d}_{\lambda}\ge d.$ Then
$r_{\lambda}\left(\mathrm{L}_{f}\right)\le\left(\frac{ed}{n}\right)^{d/2} \|f\|_{2}.$
\end{lem}

\begin{proof}
The trace of $\mathrm{L}_{f}^{*}\mathrm{L}_{f}$ is $\|f\|_{2}^{2}$.
On the other hand, by Lemma \ref{lem:Facts from representation theory}
\[
\|f\|_{2}^{2}=\mathrm{tr}\left(\mathrm{L}_{f}^{*}\mathrm{L}_{f}\right)\ge\dim\left(\lambda\right)r_{\lambda}^{2}\left(\mathrm{L}_{f}\right)\ge\left(\frac{n}{ed}\right)^{d}r_{\lambda}\left(\mathrm{L}_{f}\right)^2.
\]
\end{proof}

\subsection{Proof of our 1\% stability results}

Here we prove a version of Theorem \ref{thm:global}
that is stronger in two ways:
we consider any triple of global sets
with density $\ge n^{-O\left(1\right)}$
and we establish the product mixing phenomenon.
For a function $f\colon S_{n}\to\mathbb{R}$
we write $\tilde{f}$ for $f\cdot\sf{sign}$.
Theorem \ref{thm:global} (in contrapositive form)
is immediate from the following by setting $f=g=h=1_A$,
noting that $\tilde{f}=f$ if $f$ is supported on $A_n$.

\begin{thm}
\label{thm:stronger 1=000025 stability result} Fix $C>0$ and suppose
that $n$ is sufficiently large. Let $f,g,h\colon S_{n}\to\left\{ 0,1\right\} $
with $\mathbb{E}\left[f\right]\mathbb{E}\left[g\right]\mathbb{E}\left[h\right]\ge n^{-C}$
be $\left(i,n^{\frac{i}{4}}\right)$-relatively
global for each $i\le4\left\lceil C\right\rceil $. Then
\begin{align*}
\mathbb{E}_{\sigma,\tau\sim S_{n}}\left[f\left(\sigma\right)g\left(\tau\right)h\left(\sigma\tau\right)\right] & =\mathbb{E}\left[f\right]\mathbb{E}\left[g\right]\mathbb{E}\left[h\right]\left(1\pm n^{-0.1}\right)+\mathbb{E}\left[\tilde{f}\right]\mathbb{E}\left[\tilde{g}\right]\mathbb{E}\left[\tilde{h}\right].
\end{align*}
\end{thm}

The first step of the proof is to separate the left hand side
$\mathbb{E}_{\sigma,\tau\sim S_{n}}\left[f\left(\sigma\right)g\left(\tau\right)h\left(\sigma\tau\right)\right]$
into low degree terms and high degree ones, as in the following lemma.
\begin{lem}
\label{lem:decomposition }Let $d<\frac{n}{2}-1.$ Then
\begin{align*}
\mathbb{E}_{\sigma,\tau\sim S_{n}}\left[f\left(\sigma\right)g\left(\tau\right)h\left(\sigma\tau\right)\right]
= & \sum_{i=0}^{d}\left\langle g^{=i},\mathrm{L}_{f}h^{=i}\right\rangle
+\sum_{i=0}^{d}\left\langle \tilde{g}^{=i},\mathrm{L}_{\tilde{f}}\tilde{h}^{=i} \right\rangle\\
 & +\sum_{\lambda:\,\tilde{d_{\lambda}}>d}\left\langle g^{=\lambda},\mathrm{L}_{f}h^{=\lambda}\right\rangle .
\end{align*}
\end{lem}

\begin{proof}
As $\mathrm{L}_{f}$ commutes with the action of $S_{n}$ from the
right it preserves each $V_{=\lambda}$. We may therefore use
orthogonality to obtain the following expansion into isotypical parts:
\begin{align*}
\mathbb{E}_{\sigma,\tau\sim S_{n}}\left[f\left(\sigma\right)g\left(\tau\right)h\left(\sigma\tau\right)\right] & =\left\langle g,\mathrm{L}_{f}h\right\rangle
 =\sum_{\lambda\vdash n}\left\langle g^{=\lambda},\mathrm{L}_{f}h^{=\lambda}\right\rangle \\
 & =\sum_{\lambda:\,d_{\lambda}\le d}\left\langle g^{=\lambda},\mathrm{L}_{f}h^{=\lambda}\right\rangle +\sum_{\lambda:\,\tilde{d}_{\lambda}>d}\left\langle g^{=\lambda},\mathrm{L}_{f}h^{=\lambda}\right\rangle .
\end{align*}

\subsection*{Regrouping terms of small degree}

As $d<\frac{n}{2}-1$ at most one of $d_{\lambda^{t}},d_{\lambda}$
can be at most $d$, so
\[
\sum_{\tilde{d}_{\lambda}\le d}\left\langle g^{=\lambda},\mathrm{L}_{f}h^{=\lambda}\right\rangle =\sum_{\lambda:\,d_{\lambda}\le d}\left\langle g^{=\lambda},\mathrm{L}_{f}h^{=\lambda}\right\rangle +\sum_{\lambda:\,d_{\lambda}\le d}\left\langle g^{=\lambda^{t}},\mathrm{L}_{f}h^{=\lambda^{t}}\right\rangle .
\]
 As $V_{=i}=\sum_{\lambda:\,d_{\lambda}=i}V_{=\lambda}$ we have
\[
\sum_{\lambda:\,d_{\lambda}\le d}\left\langle g^{=\lambda},\mathrm{L}_{f}h^{=\lambda}\right\rangle =\sum_{i=0}^{d}\left\langle g^{=i},\mathrm{L}_{f}h^{=i}\right\rangle .
\]

\subsection*{Regrouping terms of small `dual' degree}

On the other hand, by Lemma \ref{lem:Facts from representation theory}
we have $g^{=\lambda^{t}}=\widetilde{\tilde{g}^{=\lambda}}$. Hence,
\begin{align*}
\sum_{d_{\lambda}\le d}\left\langle g^{=\lambda^{t}},\mathrm{L}_{f}h^{=\lambda^{t}}\right\rangle  & =\sum_{d_{\lambda}\le d}\left\langle \widetilde{\tilde{g}^{=\lambda}},\mathrm{L}_{f}\left(\widetilde{\tilde{h}^{=\lambda}}\right)\right\rangle \\
 & = \sum_{d_{\lambda}\le d} \mathbb{E}_{\pi,\sigma} \left[\tilde{f}\left(\sigma\right)\tilde{g}^{=\lambda}\left(\pi\right)\tilde{h}^{=\lambda}\left(\sigma\pi\right)\right]\\
 & =\sum_{d_{\lambda}\le d}\left\langle \mathrm{L}_{\tilde{f}}\tilde{h}^{=\lambda},\tilde{g}^{=\lambda}\right\rangle \\
 & =\sum_{i=0}^{d}\left\langle \mathrm{L}_{\tilde{f}}\tilde{h}^{=i},\tilde{g}^{=i}\right\rangle .
\end{align*}
\end{proof}
We now prove a lemma showing that the high degree terms are negligible.
\begin{lem}
\label{lem:Large degrees are negligible}
\[
\left|\sum_{\lambda:\,\tilde{d_{\lambda}}>d}\left\langle g^{=\lambda},\mathrm{L}_{f}h^{=\lambda}\right\rangle \right|\le\left(\frac{n}{e\left(d+1\right)}\right)^{-\frac{d+1}{2}}\|f\|_{2}\|g\|_{2}\|h\|_{2}.
\]
\end{lem}

\begin{proof}
By Lemma \ref{lem:trace method  } and Cauchy--Schwarz we have
\begin{align*}
\sum_{\lambda:\,\tilde{d_{\lambda}}>d}\left\langle g^{=\lambda},\mathrm{L}_{f}h^{=\lambda}\right\rangle  & \le\sum_{\lambda:\,\tilde{d_{\lambda}}>d}r_{\lambda}\left(\mathrm{L}_{f}\right)
\|h^{=\lambda}\|_{2}\|g^{=\lambda}\|_{2}\\
 & \le\left(\frac{n}{e\left(d+1\right)}\right)^{-\frac{d+1}{2}}\|f\|_{2}
 \sum_{\lambda\vdash n}\|h^{=\lambda}\|_{2}\|g^{=\lambda}\|_{2},
\end{align*}
where $\sum_{\lambda\vdash n}\|h^{=\lambda}\|_{2}\|g^{=\lambda}\|_{2}
 \le \sqrt{\sum_{\lambda\vdash n}\|h^{=\lambda}\|_{2}^{2}\sum_{\lambda\vdash n}\|g^{=\lambda}\|_{2}^{2}}
  =\|h\|_{2}\|g\|_{2}$.
\end{proof}
We now prove the theorem by
combining the bound on $\|f^{\le d}\|_{2}^{2}$ from the level-$d$ inequality
with the bounds from Lemma \ref{lem:Large degrees are negligible}
on the eigenvalues of $\mathrm{L}_{f}$ that correspond to large degrees.
\begin{proof}[Proof of Theorem \ref{thm:stronger 1=000025 stability result}]
Let $d=\left\lceil 4C\right\rceil +1.$ We have
\begin{align*}
\left\langle \mathrm{L}_{f}g,h\right\rangle  & =\left\langle \mathrm{L}_{f}h^{=0},g^{=0}\right\rangle +\left\langle \mathrm{L}_{\tilde{f}}\tilde{h}^{=0},\tilde{g}^{=0}\right\rangle \\
 & \pm\sum_{i=1}^{d}\left|\left\langle g^{=i},\mathrm{L}_{f}h^{=i}\right\rangle +\left\langle \tilde{g}^{=i},\mathrm{L}_{\tilde{f}}\tilde{h}^{=i}\right\rangle \right|\\
 & \pm\sum_{\tilde{d_{\lambda}}>d}\left|\left\langle g^{=\lambda},\mathrm{L}_{f}h^{=\lambda}\right\rangle \right|.
\end{align*}
The main terms are
$\left\langle \mathrm{L}_{f}h^{=0},g^{=0}\right\rangle
=\mathbb{E}\left[f\right]\mathbb{E}\left[g\right]\mathbb{E}\left[h\right]$
and
$\left\langle L_{\tilde{f}}\tilde{h}^{=0},\tilde{g}^{=0}\right\rangle
=\mathbb{E}\left[\tilde{f}\right]\mathbb{E}\left[\tilde{g}\right]\mathbb{E}\left[\tilde{h}\right]$.

By Lemma \ref{lem:Large degrees are negligible}, using
 $\mathbb{E}\left[f\right]\mathbb{E}\left[g\right]\mathbb{E}\left[h\right]\ge n^{-C}$,
we bound the high-degree error terms as
\[
\sum_{\tilde{d_{\lambda}}>d}\left|\left\langle g^{=\lambda},\mathrm{L}_{f}h^{=\lambda}\right\rangle \right|\le\left(\frac{n}{e\left(d+1\right)}\right)^{-\frac{d+1}{2}}\|f\|_{2}\|g\|_{2}\|h\|_{2}\le\frac{\mathbb{E}\left[f\right]\mathbb{E}\left[g\right]\mathbb{E}\left[h\right]}{n}.
\]
By Theorems \ref{thm:Level d inequality } (for $g$ and $h$) and \ref{thm:Actual radius} (for $f$)
we bound the low-degree error terms as
\begin{align*}
\sum_{i=1}^{d} \left| \left\langle \mathrm{L}_{f}h^{=i},g^{=i}\right\rangle \right|
& \le\sum_{i=1}^{d}r_{i}\left(\mathrm{L}_{f}\right)\|g^{=i}\|_{2}\|h^{=i}\|_{2}\\
 & \le\sum_{i=1}^{d}\frac{n^{3i/8}(\log n)^{O\left(1\right)}\mathbb{E}\left[f\right]\mathbb{E}\left[g\right]\mathbb{E}\left[h\right]}{n^{i/2}}\\
 & \le n^{-1/8}(\log n)^{O\left(1\right)}
 \mathbb{E}\left[f\right]\mathbb{E}\left[g\right]\mathbb{E}\left[h\right].
\end{align*}
The same bounds hold replacing $f,g,h$ on the left-hand side
by $\tilde{f}$, $\tilde{g}$, $\tilde{h}$
(which have the same globalness properties)
so the theorem follows.
\end{proof}

\subsection{The linear terms dominate}

For future reference, we conclude this section by noting that the above arguments show that
when $\mathbb{E}\left[f\right]\mathbb{E}\left[g\right]\mathbb{E}\left[h\right]$ is large, the only significant contribution to $\mathbb{E}\left[f\right]\mathbb{E}\left[g\right]\mathbb{E}\left[h\right]$
comes from the linear terms. The following is immediate
from Lemmas \ref{lem:decomposition } and \ref{lem:Large degrees are negligible} with $d=1$.
\begin{prop}
\label{Prop:Most weight is on the first two levels}
Let $f,g,h\colon A_{n}\to\left\{ 0,1\right\} $
have densities $\alpha,\beta,\gamma$ in $S_{n}$.
Then
\[ \card{\Expect{\pi,\sigma\sim S_{n}}{f(\pi)g(\sigma\circ\pi)h(\sigma)}-2\alpha\beta\gamma-2\Expect{\pi,\sigma\sim S_{n}}{f^{=1}(\pi)g^{=1}(\sigma\circ\pi)h^{=1}(\sigma)}}\le\frac{e}{n}\sqrt{\alpha\beta\gamma}.\]
\end{prop}

\section{Analysis of linear functions over the symmetric group} \label{sec:lin}

The main result of this section is our level-$1$ inequality for the pseudorandom
part of a function. This in itself will suffice to reprove Eberhard's result
(for expository purposes we will give the argument at the end of the section).
We will start by describing a canonical way to represent $f^{=1}$
as a linear combination of the dictators $x_{i\to j}.$
This canonical representation $f^{=1}=\sum a_{ij}x_{i\to j}$
will naturally lead to a decomposition of $f^{=1}$
as a sum of its random part
$f_{\mathrm{rand}}:=\sum_{\left|a_{ij}\right|<\epsilon}a_{ij}\left(x_{i\to j}-\frac{1}{n}\right)$
and its structural part $f_{\mathrm{struc}}.$
Our level-$1$ inequality will bound $\|f_{\mathrm{rand}}\|_{2}$
by $\epsilon\|f\|_{2}$ up to logarithmic factors.

\subsection{The normalized form of linear functions}

We say that a linear function $\sum_{i,j}a_{ij}x_{i\to j}$ is in
\emph{normalized form} if for each $i$ we have $\sum_{j=1}^{n}a_{ij}=0$
and for each $j$ we have $\sum_{i=1}^{n}a_{ij}=0$. Every linear
function in normalized form has zero expectation, i.e.\ is in $V_{=1}$.
We will soon show the converse, i.e.\ that every $f\in V_{=1}$ has a normalized form.
First we give a simple formula for the inner product between two linear functions,
which holds when at least one of them is in normalized form.
\begin{lem}
\label{lem:Parseval } Let $f=\sum_{i,j\in\left[n\right]}a_{ij}x_{i\to j}$
be in normalized form. Let $g=\sum_{i,j}b_{ij}x_{i\to j}$ be an arbitrary
linear function. Then
\[
\left\langle f,g\right\rangle =\frac{\sum_{i,j}a_{ij}b_{ij}}{n-1}.
\]
\end{lem}

\begin{proof}
Consider the linear functionals $\varphi,\psi\colon\mathbb{R}^{n\times n}\to\mathbb{R}$
given by
\[
\varphi\left(\left(b_{ij}\right)_{i,j}\right)=\frac{1}{n-1}\sum_{i,j}a_{ij}b_{ij}
\]
 and
\[
\psi\left(\left(b_{ij}\right)_{i,j}\right)=\left\langle f,\sum b_{ij}x_{i\to j}\right\rangle .
\]
 As both $\varphi,\psi$ are linear it is enough to show that $\varphi=\psi$
on a basis. Hence, it is sufficient to prove the lemma when $g=x_{i\to j}.$
There we may use the fact that $f$ is in normalized form to deduce
that
\begin{align*}
\left\langle f,g\right\rangle  & =\frac{1}{n}a_{ij}+\frac{1}{n\left(n-1\right)}\sum_{i'\ne i,j'\ne j}a_{i'j'}\\
 & =\frac{1}{n}a_{ij}-\frac{1}{n\left(n-1\right)}\sum_{j'\ne j}a_{ij'}\\
 & =\frac{1}{n}a_{ij}+\frac{1}{n\left(n-1\right)}a_{ij}\\
 & =\frac{1}{n-1}a_{ij}.
\end{align*}
 This completes the proof.
\end{proof}
We are now ready to show that every function in $V_{=1}$ has a normalized
form. In fact, we give an explicit formula for the coefficients of
each $f^{=1}\in V_{=1}.$
\begin{lem}
\label{lem:Normalised form of f=00003D1} Let $f\colon S_{n}\to\mathbb{R}$.
Let
\[
a_{ij}=\frac{n-1}{n}\left(\mathbb{E}\left[f_{i\to j}\right]-\mathbb{E}\left[f\right]\right).
\]
 Then $\sum_{i,j}a_{ij}x_{i\to j}$ is a normalized form of $f^{=1}$.
\end{lem}

\begin{proof}
First we note that $\sum_j \mb{E}[f_{i \to j}] = n\mb{E}f$ for each $i$,
so $\sum_{i,j}a_{ij}x_{i\to j}$ is a linear function in normalized form.
As $f^{=1}$ is the projection of $f$ onto the space $V_{=1}$
of linear functions with expectation 0, to prove the lemma it suffices
to show that for each linear function $g$ with $\mathbb{E}\left[g\right]=0$
we have $\left\langle f,g\right\rangle =\left\langle \sum_{i,j}a_{ij}x_{i\to j},g\right\rangle .$

By linearity, it is enough to show this when $g=x_{i\to j}-\frac{1}{n}$,
as such functions $g$ span $V_{=1}$. For such $g$ we have
\[
\left\langle f,g\right\rangle =\frac{1}{n}\mathbb{E}\left[f_{i\to j}\right]-\frac{1}{n}\mathbb{E}\left[f\right]=\frac{1}{n-1}a_{ij}.
\]
 On the other hand, by Lemma \ref{lem:Parseval } we have
\[
\left\langle \sum_{i,j}a_{ij}x_{i\to j},g\right\rangle =\frac{1}{n-1}a_{ij}
\]
 as $\sum_{i,j}a_{ij}x_{i\to j}$ is in normalized form.
\end{proof}

For $f=\sum_{i,j}a_{ij}x_{i\to j}$ in normalized form,
Lemma \ref{lem:Parseval } gives the Parseval formula
$\|f\|_2^2 = (n-1)^{-1} \sum_{i,j} a_{ij}^2$.
For any linear function, not necessarily in normalized form,
we still have the following upper bound,
which has the same form up to a constant factor.
\begin{lem}
\label{lem:one sided pars} Let $g=\sum_{i,j}a_{i,j}\left(x_{i\to j}-\frac{1}{n}\right)$
be a function in $V_{=1}.$ Then $\|g\|_{2}^{2}\le\frac{8}{n}\sum_{i,j}a_{ij}^{2}.$
\end{lem}

\begin{proof}
Computing, we get
\begin{align*}
\norm g_{2}^{2} & =\sum\limits _{i,j}a_{i,j}^{2}\frac{1}{n}\left(1-\frac{1}{n}\right)-\sum\limits _{i,j}\sum\limits _{i'\neq i}a_{i,j}a_{i',j}\frac{1}{n^{2}}-\sum\limits _{i,j}\sum\limits _{j'\neq j}a_{i,j}a_{i,j'}\frac{1}{n^{2}} \\
& +\sum\limits _{i,j}\sum\limits _{\substack{i'\neq i,j'\neq j}}a_{i,j}a_{i',j'}\frac{1}{n^{2}(n-1)}.
\end{align*}
We bound each term on the right hand side separately. Using $2\card{ab}\leq a^{2}+b^{2}$,
each of the sums on the right hand side is $\le\frac{2}{n}\sum\limits _{i,j}a_{i,j}^{2}$,
and so $\norm g_{2}^{2}\le\frac{8}{n}\sum\limits _{i,j}a_{i,j}^{2}$.
\end{proof}
Using Lemma \ref{lem:Parseval } we can now derive a useful formula
for the linear term in the count for the number of products
in terms of the coefficient matrices of the normalised forms.
\begin{lem}
\label{lem:convolution and parseval}
Let $f=\sum_{i,j}a_{ij}x_{i\to j},g=\sum_{i,j}b_{ij}x_{i\to j},h=\sum_{i,j}c_{ij}x_{i\to j}$
all be linear functions in normalized form. Let their coefficient
matrices be defined by $M_{f}=\left(a_{ij}\right)_{i,j},M_{g}=\left(b_{ij}\right)_{i,j},$
and $M_{h}=\left(c_{ij}\right)_{i,j}.$ Then
\[
\mathbb{E}_{\sigma,\tau\sim S_{n}}\left[f(\sigma)g(\tau)h(\sigma\tau)\right]=\frac{1}{\left(n-1\right)^{2}}\left\langle M_{g}M_{f},M_{h}\right\rangle .
\]
\end{lem}

\begin{proof}
We have
\begin{align*}
R_{\tau}h\left(\sigma\right) & =\sum_{i,j\in\left[n\right]}c_{ij}x_{i\to j}\left(\sigma\tau\right)=\sum_{i,j}c_{ij}\sum_{k}x_{i\to k}\left(\tau\right)x_{k\to j}\left(\sigma\right)\\
 & =\sum_{k,j}d_{kj}\left(\tau\right)x_{k\to j}\left(\sigma\right),
\end{align*}
where $d_{kj}\left(\tau\right)=\sum_{i}c_{ij}x_{i\to k}\left(\tau\right).$
By Lemma \ref{lem:Parseval } we therefore have
\begin{align*}
\mathrm{L}_{f}h & \left(\tau\right)=\left\langle f,R_{\tau}h\right\rangle
=\frac{1}{n-1}\sum_{i,j\in\left[n\right]}a_{ij}d_{ij}(\tau)
=\frac{1}{n-1}\sum_{i,j,k}a_{ij}c_{kj}x_{k\to i}\left(\tau\right)=\\
 & =\frac{1}{n-1}\sum_{i,j} e_{ij}x_{i\to j}\left(\tau\right),
\end{align*}
 where $e_{ij}=\sum_{k}a_{jk}c_{ik}.$
 We deduce that
\begin{align*}
\mathbb{E}\left[f\left(\sigma\right)g\left(\tau\right)h\left(\sigma\tau\right)\right]
& =\left\langle g,\mathrm{L}_{f}h\right\rangle =\frac{1}{(n-1)^2}\sum_{i,j} e_{ij}b_{ij}\\
 & =\frac{1}{\left(n-1\right)^{2}}\sum_{i,j,k} b_{ij}a_{jk}c_{ik}
  =\frac{1}{\left(n-1\right)^{2}}\left\langle M_{g}M_{f},M_{h}\right\rangle .
\end{align*}
\end{proof}

\subsection{Global hypercontractivity and the level-$1$ inequality}

The following lemma shows that Theorem \ref{thm:hyp} may be applied
to linear functions with small coefficients. The lemma is applicable
for $f_{\mathrm{rand}}$ defined above.
\begin{lem}
\label{lem:globalness of a function with small coefficients}
Let $g=\sum_{i,j}a_{ij}\left(x_{i\to j}-\frac{1}{n}\right) \in V_{=1}$
with $|a_{ij}|<\eps$ for all $i,j$.
Then $g$ is $\left(2,\epsilon'\right)$-global
for $\epsilon'=9\epsilon+\sqrt{\frac{8}{n-2}\sum_{i,j}a_{ij}^{2}}.$
\end{lem}

\begin{proof}
We need to show $\|g_{I \to J}\|_2 \le \epsilon'$
for any restriction with $|I|=|J| \le 2$.
By averaging, it suffices to consider $|I|=|J|=2$.
Take distinct $i,j,k,\ell$,
and consider the restriction $i\rightarrow j,k\rightarrow\ell$, corresponding
to the duumvirate $\mathcal{U}_{i\to j,k\to l}.$
We apply the bound
\[
\norm{g_{i\rightarrow j,k\rightarrow\ell}}_{2}\leq\left|L\right|
+\left\Vert \tilde{g}_{i\to j,k\to l}\right\Vert _{2},
\]
where we define $\tilde{g}\colon\mathcal{U}_{i\to j,k\to l} \to \mb{R}$ by
\[
\tilde{g}(\pi)=\sum\limits _{t\neq i,j}\sum\limits _{q\neq k,\ell}a_{t,q}\left(x_{t\to q}-\frac{1}{n-2}\right),
\]
and let
\begin{align*}
L=g_{i\to j,k\to l}-\tilde{g} & =\left(1-\frac{1}{n}\right)(a_{i,j}+a_{k,\ell}),\\
 & -\frac{1}{n}\sum\limits _{t\neq j}a_{i,t}-\frac{1}{n}\sum\limits _{t\neq\ell}a_{k,t}-\frac{1}{n}\sum\limits _{t\neq i}a_{t,j}-\frac{1}{n}\sum\limits _{t\neq k}a_{t,\ell}\\
 & +\sum\limits _{t\neq i,j}\sum\limits _{q\neq k,\ell}a_{t,q}\left(\frac{1}{n-2}-\frac{1}{n}\right).
\end{align*}
By the triangle inequality we have $\card L\leq9\epsilon$.
For the second term above, we may use the aforementioned identification between
$\mathcal{U}_{i\to j,k\to l}$ and $S_{n-2}$ to apply Lemma \ref{lem:one sided pars},
deducing that
\[
\norm{\tilde{g}_{i\rightarrow j,k\rightarrow\ell}}_{2}^{2}
\le\frac{8}{n-2}\sum\limits _{t\neq i,j}\sum\limits _{q\neq k,\ell}a_{t,q}^{2}.
\]
Thus we obtain the required bound
$\norm{g_{i\rightarrow j,k\rightarrow\ell}}_{2} \le \epsilon'$.
\end{proof}

\subsection{A level-1 inequality for the pseudorandom part}

We now show the desired upper bound on the $L^2$-norm
$\|f_{\mathrm{rand}}\|_{2}$ of the pseudorandom part of $f^{=1}$;
in fact, we bound the right hand side of the bound
$\norm{f_{\mathrm{rand}}}_{2}^{2}\le 8X^2$
from Lemma \ref{lem:one sided pars},
where $X^2$ is as in the following statement.

\begin{lem}
\label{lem:level-1 inequality}
Let $\epsilon\in\left(0,\frac{1}{2}\right)$
and $f\colon S_{n}\to\left\{ 0,1\right\} $ with $\mathbb{E}\left[f\right]\le\frac{1}{2}.$
Write $f^{=1} = \sum_{i,j}a_{ij}x_{i\to j}$ in normalized form and let
$f_{\mathrm{rand}} = \sum_{i,j}a_{ij}\left(x_{i\to j}-\frac{1}{n}\right)1_{|a_{ij}|<\eps}$.
Denote $\epsilon''=\max\left(\epsilon,\mathbb{E}\left[f\right]\right).$
Then
\[
X^2 := \frac{1}{n-1} \sum_{i,j} a_{ij}^{2} 1_{|a_{ij}|<\eps}
\le\mathbb{E}\left[f\right] \epsilon''\log^{O\left(1\right)}\left(\frac{1}{\epsilon''}\right).
\]
\end{lem}

\begin{proof}
We note that $f_{\mathrm{rand}} \in V_{=1}$ and by Lemma \ref{lem:Parseval } we have
\[
X^2 = \left\langle f_{\mathrm{rand}},f\right\rangle
=\left\langle f_{\mathrm{rand}},f^{=1}\right\rangle.
\]
Let $q=10\log\left(1/\epsilon''\right)$ and $\epsilon'=9\epsilon+3\|f_{\mathrm{rand}}\|_{2}$,
so that $f_{\mathrm{rand}}$ is $\left(2,\epsilon'\right)$-global
by Lemma \ref{lem:globalness of a function with small coefficients}.
Applying Theorem~\ref{thm:hyp} we obtain
\begin{align*}
X^2 & \le 8\inner{f_{\mathrm{rand}}}{f}
\leq\norm{f_{\mathrm{rand}}}_{q}\norm f_{q/(q-1)}\\
 & \leq q^{O\left(1\right)}(\eps')^{1-\frac{2}{q}}\norm{f_{\mathrm{rand}}}_{2}^{\frac{2}{q}}\norm f_{q/(q-1)}\\
 & =q^{O\left(1\right)}(\eps')^{1-\frac{2}{q}}\norm{f_{\mathrm{rand}}}_{2}^{\frac{2}{q}}\norm f_{2}^{2-\frac{2}{q}},
\end{align*}
where we used the fact that $f$ is $\left\{ 0,1\right\} $-valued.
Using $\norm{f_{\mathrm{rand}}}_{2}^{2}\le 8X^2$
by Lemma \ref{lem:one sided pars} and rearranging we get
\[
X^2 \le q^{O\left(1\right)}(\eps')^{\frac{q-2}{q-1}} \mb{E}[f].
\]
We next consider two cases according to which of $\epsilon$ and $X$ is larger.
\begin{enumerate}
\item If $X \leq\epsilon$ then $\eps'\le 20\epsilon$, so we obtain
\[
X^2 \le q^{O\left(1\right)}\epsilon^{\frac{q-2}{q-1}} \mb{E}[f].
\]
\item If $X>\epsilon$ then $\eps'\le 20X$,
so $X^2 \le q^{O\left(1\right)} X^{\frac{q-2}{q-1}} \mb{E}[f]$, yielding
\[
X^2  \le q^{O\left(1\right)} \mb{E}[f]^{2-\frac{2}{q}}.
\]
\end{enumerate}
In both cases the lemma follows by plugging in $q=10\log\left(1/\epsilon''\right)$.
\end{proof}

\subsection{Recovering Eberhard's result}

For expository purposes, we will now cash in on our level-$1$ inequalities
and reprove Eberhard's result (up to the polylog factor).

We will repeatedly use the following upper bound
on $\left|\left\langle MN,S\right\rangle \right|$
for three matrices $M,N,S.$
\begin{lem}
\label{lem:Csing matrices} Let $M,N,S\in\mathbb{R}^{n\times n}.$
Then we have
\[
\left|\left\langle MN,S\right\rangle \right|\le\|M\|_{2}\|N\|_{2}\|S\|_{2}.
\]
\end{lem}

\begin{proof}
Write $Ne_{i}=v_{i},Se_{i}=u_{i}$. By Cauchy--Schwarz
\begin{align*}
\left|\left\langle MN,S\right\rangle \right| & =\left|\sum_{i=1}^{n}\left\langle Mv_{i},u_{i}\right\rangle \right|\le\sum_{i=1}^{n}\|M\|_{2}\|v_{i}\|_{2}\|u_{i}\|_{2}\\
 & =\|M\|_{2}\sum_{i=1}^{n}\|v_{i}\|_{2}\|u_{i}\|_{2}\le\|M\|_{2}\sqrt{\sum_{i=1}^{n}\|v_{i}\|_{2}^{2}}\sqrt{\sum_{i=1}^{n}\|u_{i}\|_{2}^{2}}.\\
 & =\|M\|_{2}\|N\|_{2}\|S\|_{2}.
\end{align*}
\end{proof}
\begin{lem}
\label{lem:one sided Eberhard} Let $A,B,C\subseteq A_{n}$ and write
$\alpha=\frac{\left|A\right|}{n!}$, $\beta=\frac{\left|B\right|}{n!}$,
$\gamma=\frac{\left|C\right|}{n!}.$ Then
\begin{align}
& \Pr_{\sigma,\tau\sim S_{n}}\left[\sigma\in A,\tau\in B,\sigma\tau\in C\right]  \ge2\alpha\beta\gamma-\frac{32\sqrt{\alpha\beta\gamma}}{n}-\alpha\beta\gamma\frac{\log^{O\left(1\right)}\left(\alpha\beta\gamma\right)}{\sqrt{n}}\nonumber \\
 & \qquad -\alpha\log^{O\left(1\right)}\left(\frac{1}{\alpha}\right)\sqrt{\frac{\beta\gamma}{n}}-\beta\log^{O\left(1\right)}\left(\frac{1}{\beta}\right)\sqrt{\frac{\gamma\alpha}{n}}-\gamma\log^{O\left(1\right)}\left(\frac{1}{\gamma}\right)\sqrt{\frac{\alpha\beta}{n}}.\label{eq:one sided eberhard}
\end{align}
\end{lem}

\begin{proof}
Write $f=1_{A}$, $f^{=1}=\sum a_{ij}x_{i\to j}$, $\matA=\left(a_{ij}\right)$,
$\matA_{-}=\left(a_{ij}1_{a_{ij}<0}\right)$,
$\matA_{+}=\left(a_{ij}1_{a_{ij}>0}\right)$.
We use analogous notation for $g=1_{B}$ and $h=1_{C}$,
with $g=\sum b_{ij}x_{i\to j}$ and $h=\sum c_{ij}x_{i\to j}$.

Then by Proposition \ref{Prop:Most weight is on the first two levels}
\begin{align*}
& \Pr_{\sigma,\tau\sim A_{n}}\left[\sigma\in A,\tau\in B,\sigma\tau\in C\right]
= \mathbb{E}_{\sigma,\tau\sim S_{n}}\left[f\left(\sigma\right)g\left(\tau\right)h\left(\sigma\tau\right)\right]\\
 & \ge \qquad 2\alpha\beta\gamma+2\mathbb{E}\left[f^{=1}\left(\sigma\right)g^{=1}\left(\tau\right)h^{=1}\left(\sigma\tau\right)\right]-\frac{e\sqrt{\alpha\beta\gamma}}{n}.
\end{align*}
 By Lemma \ref{lem:convolution and parseval} we have
\[
\mathbb{E}\left[f^{=1}\left(\sigma\right)g^{=1}\left(\tau\right)h^{=1}\left(\sigma\tau\right)\right]=\frac{\left\langle \matB\matA,\matC\right\rangle }{\left(n-1\right)^{2}}.
\]
 By our level $1$ inequality (Lemma \ref{lem:level-1 inequality}) we have
\[
\|\matA_{-}\|_{2}^{2}\le n\alpha^{2}\log^{O\left(1\right)}\left(1/\alpha\right),
\]
with analogous statements for $\matB_{-}$ and $\matC_{-}$.
By our Parseval lemma (Lemma \ref{lem:Parseval }) we have
\[
\|\matA_{+}\|_{2}^{2}\le\|\matA\|_{2}^{2}=\left(n-1\right)\alpha,
\]
 and similarly for $\matB$ and $\matC.$

We may now expand $\left\langle \matB\matA,\matC\right\rangle $ by
writing $\matA=\matA_{+}+\matA_{-}$ and similarly for $B$ and $C$.
After discarding the terms with a positive contribution
to $\left\langle \matB\matA,\matC\right\rangle $
we are left with the four terms
\[
\left\langle \matB_{+}\matA_{-},\matC_{+}\right\rangle ,\left\langle \matB_{-}\matA_{+},\matC_{-}\right\rangle ,\left\langle \matB_{+}\matA_{+},\matC_{-}\right\rangle ,\left\langle \matB_{-}\matA_{-},\matC_{-}\right\rangle .
\]
 The lemma now follows from Lemma \ref{lem:Csing matrices}.
\end{proof}
When $\left(A,B,C\right)$ is product-free the above
immediately implies the following bounds on their densities,
as if $\min\left(\alpha\beta,\beta\gamma,\gamma\alpha\right)\ge\frac{\log^{R}n}{n}$
for sufficiently large $R$ then all terms bar $2\alpha\beta\gamma$
in the right hand side of (\ref{eq:one sided eberhard}) are $o\left(\alpha\beta\gamma\right)$.

\begin{cor}
\label{cor:eberhard}Let $\left(A,B,C\right)$ be product-free. Write
$\alpha=\frac{\left|A\right|}{n!},\beta=\frac{\left|B\right|}{n!}$
and $\gamma=\frac{\left|C\right|}{n!}.$ Then
\[
\min\left(\alpha\beta,\beta\gamma,\gamma\alpha\right)\le\frac{\log^{O\left(1\right)}n}{n}.
\]
\end{cor}

For future reference, we also note the following slightly stronger bound
for the regime $\bB, \gG \ge \eps > n^{-o(1)}$, where we can replace
the factor $\log^{O\left(1\right)}n$ by
$\log^{O\left(1\right)}\left(\frac{1}{\epsilon}\right)$.

\begin{cor}
\label{cor:unbalanced eberhard}
Let $\left(A,B,C\right)$ be product-free.
Write $\alpha=\frac{\left|A\right|}{n!},\beta=\frac{\left|B\right|}{n!}$
and $\gamma=\frac{\left|C\right|}{n!}.$
Suppose $\min\left(\beta,\gamma\right)\ge\epsilon$ with $\epsilon \in (0,1/2)$.
Then
\[
\alpha\le\frac{1}{\epsilon n}\log^{O\left(1\right)}\left(\frac{1}{\epsilon}\right).
\]
\end{cor}

\section{Dictatorial structure} \label{sec:dictate}

In this section we will expose the dictatorial structure
of product-free sets and triples that are not too sparse.
Throughout the remainder of the paper we adopt the following notation.
We let $A,B,C\subseteq S_{n}$, let $f=1_{A}, g=1_{B}, h=1_{C}$, and
\[
f^{=1}=\sum a_{ij}x_{i\to j}, \quad
g^{=1}=\sum_{i,j}b_{ij}x_{i\to j}, \quad
h^{=1}=\sum_{i,j}c_{ij}x_{i\to j}.
\]
We write $\matA=\left(a_{ij}\right),$ $\matB=\left(b_{ij}\right)$
and $\matC=\left(c_{ij}\right).$

Proposition \ref{Prop:Most weight is on the first two levels} shows
that in order to understand
\[
\mathbb{E}_{\sigma,\tau\sim S_{n}}\left[f\left(\sigma\right)g\left(\tau\right)h\left(\sigma\tau\right)\right]
\]
it is sufficient to understand the linear part
\[
\mathbb{E}_{\sigma,\tau\sim S_{n}}\left[f^{=1}\left(\sigma\right)g^{=1}\left(\tau\right)h^{=1}\left(\sigma\tau\right)\right],
\]
for which Lemma \ref{lem:convolution and parseval} gives the formula
\[
\mathbb{E}_{\sigma,\tau\sim S_{n}}\left[f^{=1}\left(\sigma\right)g^{=1}\left(\tau\right)h^{=1}\left(\sigma\tau\right)\right]=\frac{1}{\left(n-1\right)^{2}}\left\langle \matB\matA,\matC\right\rangle .
\]
We will decompose the matrix $\matA$ (and similarly $\matB,\matC$) into three parts:
\begin{enumerate}
\item The matrix $\matA_{-}$ contains the negative coefficients of $\matA$,
and so represents the negative correlations that $A$ has with dictators.
\item The matrix $\Apseudo$ contains the small positive coefficients of $\matA$,
and so represents the pseudorandom part of $f^{=1}.$
\item The matrix $\Astruc$ contains the large coefficients of $\matA$,
and so corresponds to the dictators with which $A$ is heavily correlated.
\end{enumerate}
We expand $\left\langle \matB\matA,\matC\right\rangle $
according to this decomposition and show that most
of the negative contributions come from the terms
\[
\left\langle \Bstruc\matA_{-},\Cstruc\right\rangle
+\left\langle \matB_{-}\Astruc,\Cstruc\right\rangle
+\left\langle \Bstruc\Astruc,\matC_{-}\right\rangle,
\]
namely those compatible triples of dictators
for which two of the matrices have a strong positive correlation
and the third has a negative correlation.

\subsection{Parameters}

The following parameters will be used throughout the remainder of the paper.
We let
\[  \mu\left(A\right)=\alpha, \quad
\mu\left(B\right)=\beta, \quad
\mu\left(C\right)=\gamma. \]
We fix $R$ much larger than all the absolute constants
implicitly appearing in our $O\left(1\right)$ notation
and suppose that $n$ is sufficiently large with respect to $R$.
We let
\[ \delta=\log^{-R}n, \quad
\epsilon_{A}=n\delta\alpha\min\left(\beta,\gamma\right), \quad
\epsilon_{B}=n\delta\beta\min\left(\alpha,\gamma\right), \quad
\epsilon_{C}=n\delta\gamma\min\left(\alpha,\beta\right).
\]
Note that in our exact result, which is the case of most interest,
we have $A=B=C$ and $\alpha=\beta=\gamma=\Theta\left(\frac{1}{\sqrt{n}}\right)$,
so $\epsilon_{A}=\epsilon_{B}=\epsilon_{C}=\Theta\left(\delta\right)$.

\subsection{Our decomposition}

For a matrix $M=\left(m_{ij}\right)$ and an interval $I\subseteq\mathbb{R}$
we write $M_{I}=\left(a_{ij}1_{a_{ij}\in I}\right).$
As mentioned above,
our idea is to decompose our matrix $\matA$ as the sum
\[
\mathrm{A}=\matA_{-}+\Astruc+\Apseudo,
\]
where $\matA_{-}=\matA_{\left(-\infty,0\right)},$
$\Apseudo=\matA_{\left(0,\epsilon_{A}\right)}$
and $\Astruc=\matA_{[\epsilon_{A},\infty)}$.
We decompose $\matB$ and $\matC$ similarly with
$\epsilon_{B}$ and $\epsilon_{C}$ replacing $\epsilon_{A}$.

As mentioned in the introduction, the key to our approach is to combine
Lemma \ref{lem:Csing matrices} with the following upper bounds on
$\|\Apseudo\|_{2}^{2}$ and $\|\matA_{-}\|_{2}^{2}$,
which will follow easily from our level-$1$ inequalities in the previous section.
\begin{lem}
\label{Lem:Upper bounding two norms of matrices} Let $A\subseteq S_{n}$
and let $\mathrm{A}_{-},\Astruc,\Apseudo$ as above. Then
\begin{enumerate}
\item $\|\mathrm{A}_{-}\|_{2}^{2}\le n\alpha^{2}\log^{O\left(1\right)}\left(1/\alpha\right)$.
\item $\|\Apseudo\|_{2}^{2}\le n\alpha\epsilon_{A}\log^{O\left(1\right)}\left(1/\alpha\right)$.
\item $\|\Astruc\|_{2}^{2}\le n\alpha.$
\end{enumerate}
Analogous statements hold for $B$ and $C$.
\end{lem}

\begin{proof}
Statements (1) and (2) follow immediately from Lemma \ref{lem:level-1 inequality}.
Statement (3) follows from Lemma \ref{lem:Parseval }
and the fact that $\|f^{=1}\|_{2}^{2}\le\|f\|_{2}^{2}=\alpha$.
\end{proof}
We are now ready to show that the only significant negative contributions
to $\left\langle \matB\matA,\matC\right\rangle $ come from two structure
matrices and one negative coefficient matrix.
\begin{lem}
\label{lem:Being left with struc} We have
\begin{align*}
\left\langle \matB\matA,\matC\right\rangle  & \ge\left\langle \Bstruc\matA_{-},\Cstruc\right\rangle +\left\langle \matB_{-}\Astruc,\Cstruc\right\rangle +\left\langle \Bstruc\Astruc,\matC_{-}\right\rangle \\
 & +\sqrt{\delta}n^{2}(\log n)^{O\left(1\right)}\alpha\beta\gamma.
\end{align*}
\end{lem}

\begin{proof}
We expand the left hand side according to the decomposition
$\mathrm{A}=\matA_{-}+\Astruc+\Apseudo$ and similarly for $\matB$, $\matC$.
For a lower bound we can discard all the terms that involve
an even number of $\matA_{-},\matB_{-},\matC_{-}$ as those have a
non-negative contribution to $\left\langle \matB\matA,\matC\right\rangle$.
For the remaining terms not listed on the right hand side above we may
apply Lemmas \ref{lem:Csing matrices} and \ref{Lem:Upper bounding two norms of matrices}
to deduce that they have absolute value at most
\begin{align*}
 & <n^{1.5}(\log n)^{O\left(1\right)}\left(\left(\sqrt{\epsilon_{B}}+\sqrt{\epsilon_{C}}\right)\alpha\sqrt{\beta\gamma}\right)\\
 & +n^{1.5}(\log n)^{O\left(1\right)}\left(\left(\sqrt{\epsilon_{A}}+\sqrt{\epsilon_{C}}\right)\beta\sqrt{\alpha\gamma}\right)\\
 & +n^{1.5}(\log n)^{O\left(1\right)}\left(\left(\sqrt{\epsilon_{A}}+\sqrt{\epsilon_{B}}\right)\gamma\sqrt{\alpha\beta}\right)\\
 & +\alpha\beta\gamma n^{1.5} (\log n)^{O\left(1\right)}\\
 & \le (\log n)^{O\left(1\right)} \left(\sqrt{\delta}+\frac{1}{\sqrt{n}}\right)n^{2}\alpha\beta\gamma.\\
 & =\sqrt{\delta}n^{2}(\log n)^{O\left(1\right)} \alpha\beta\gamma
\end{align*}
 Here, the first three lines corresponds to terms such as $\left\langle \Bpseudo\matA_{-},\Cstruc\right\rangle $
and $\left\langle \Bpseudo\matA_{-},\Cpseudo\right\rangle $, while
the fourth line corresponds to the term $\left\langle \matB_{-}\matA_{-},\matC_{-}\right\rangle $.
The second inequality follows by plugging in the values of $\epsilon_{A},\epsilon_{B},\epsilon_{C}.$
\end{proof}

\section{Star structure} \label{sec:star}

In this section we will refine the dictatorial structure established
in the previous section to extract a strong star structure
that explains how some set or triple in $A_n$
can be quite dense yet product-free.
These stability results will then be refined by bootstrapping
arguments in the next section to deduce our exact and strong stability results.

\subsection*{Equivalence and inversion}

The equation $ab=c$ can be written in $6$ equivalent ways, e.g.\
we may write $ca^{-1}=b$ and $b^{-1}a^{-1}=c^{-1}$.
Thus if the triple $\left(A,B,C\right)$ is product-free
then we have $6$ equivalent product-free triples such as
$\left(C,A^{-1},B\right)$ and $\left(B^{-1},A^{-1},C^{-1}\right)$.
The structure explaining this product-freeness
may appear in any of $6$ different forms,
so to avoid cumbersome statements,
we will say that a certain structural statement
for $(A,B,C)$ holds \emph{up to equivalence}
if it holds when $(A,B,C)$ is replaced
by one of its $6$ equivalent triples.
Similarly, for a single product-free set $A$,
the structural statement for $A$ may apply to $A$ or $A^{-1}$,
so we will say that it holds \emph{up to inversion}.

\subsection{Goals of this section}

Our first main result of this section will show that
any product-free set has a strong star structure
under a fairly mild assumption on its density
(recall that $\delta^{-1} = \log^R n$).
\begin{prop}
\label{prop:Explanation for product-freeness}
Suppose that $A$ is product-free
with $\mu\left(A\right)\ge\delta^{-2}n^{-2/3}$.
Then up to inversion there exist $x\in\left[n\right]$
and $I\subseteq\left[n\right]$ such that
\[
\mu\left(A\setminus1_{x\to I}\right)\le O(\dD^{-2}) n^{-\frac{2}{3}}.
\]
 Moreover, for each $i\in I$ we have $\mu\left(A_{x\to i}\right)\ge n^{-1/3}.$
\end{prop}

Our second main result of the section describes the star structure
for product-free triples under mild density assumptions:
up to equivalence $B$ and $C$ must be strongly correlated
with stars at some common vertex $x$.

\begin{prop}
\label{prop:Explanation for cross product-freeness}
Suppose that $\left(A,B,C\right)$ is product-free with
\begin{equation} \label{eq:5.2}
\alpha\min\left(\beta,\gamma\right)^{2},\beta\min\left(\alpha,\gamma\right)^{2},\gamma\min\left(\alpha,\beta\right)^{2}\ge\delta^{-5}n^{-2}.
\end{equation}
Then up to equivalence there exist $x\in\left[n\right]$
and $I,J\subseteq\left[n\right]$ such that the following hold:
\begin{enumerate}
\item We have $\mu_{B}\left(1_{x\to I}\right)\ge\frac{1}{100}$
and $\mu_{C}\left(1_{x\to J}\right)\ge\frac{1}{100}$.
\item For each $i\in I$ and $j\in J$ we have $\mu\left(B_{x\to i}\right)\ge\epsilon_{B}$
and $\mu\left(C_{x\to j}\right)\ge\epsilon_{C}.$
\end{enumerate}
\end{prop}

\subsection{Associated stars}

We introduce some further notation that will be used throughout
the remainder of the paper to describe the stars associated
to the structured parts of $A$, $B$ and $C$. For each $i \in [n]$,
we define the \emph{associated star} for $(A,i)$ by
\[ S_{A}(i) = 1_{i\to L_{A}\left(i\right)},
\quad \text{ where } L_{A}\left(i\right)=\{j: a_{ij}>\epsilon_{A}\}. \]
(Recall that $\epsilon_{A},\epsilon_{B},\epsilon_{C}$
were defined in the previous section.)

Similarly,
we define the \emph{associated inverse star} for $(A,i)$ by
\[ S_{A}'(i) = 1_{L_{A}'\left(i\right) \to i},
\quad \text{ where } L_{A}'\left(i\right)=\{j: a_{ji}>\epsilon_{A}\}. \]
We write
\[
s_{A}\left(i\right)=\frac{\sum_{j\in L_{i}\left(A\right)}a_{ij}}{n-1}
\quad \text{ and } \quad
s_{A}'\left(i\right)=\frac{\sum_{j\in L_{A}'\left(i\right)}a_{ji}}{n-1}.
\]

We may interpret $s_{A}\left(i\right)$ and $s'_{A}\left(i\right)$
combinatorially as the correlation between $A$ and the corresponding
associated (inverse) star by noting that
\[
s_{A}\left(i\right)=\mu\left(A\cap S_{A}\left(i\right)\right)-\mu\left(A\right)\mu\left(S_{A}\left(i\right)\right)
\]
and similarly for $s_{A}'\left(i\right)$.
We also define the corresponding notions for $B$ and $C$ similarly.
We say that an associated star $S_{A}\left(i\right)$
is \emph{small }if $s_{A}\left(i\right)\le\delta\mu\left(A\right)$,
or otherwise we say that it is \emph{large}.
We use this terminology for associated stars
and associated inverse stars for each of $A,B,C$.

\subsection{Overview of proof} \label{sec:starover}

We will now give an overview of the arguments used to
extract star structure from the inequalities discussed above.
For simplicity in the overview we will concentrate
on the case of a single product-free set $A$,
which is analysed by applying the inequalities with $A=B=C$.

For terms such as $\left\langle \Astruc\matA_{-},\Astruc\right\rangle$
we use the fact that each coefficient of $\mathrm{A}_{-}$ is at least $-\alpha$,
which provides the following lower bound in terms of the associated stars:
\[
\frac{1}{\left(n-1\right)^{2}}\left\langle \Astruc\matA_{-},\Astruc\right\rangle \ge-\alpha\sum_{i=1}^{n}s_{A}\left(i\right)^{2}.
\]
Summing over all the similar terms,
we thus reduce our goal to a lower bound
\begin{equation}
-\alpha\sum_{i=1}^{n} \left( s_{A}\left(i\right)^{2} + s_{A}\left(i\right)s_{A}'\left(i\right)
+ s_{A}'\left(i\right)^{2} \right)
\ge -\alpha^3 + X,
\label{eq:associated stars}
\end{equation}
where $X$ dominates the error terms unless $A$ has strong star structure.

There will be two main ingredients in this bound.
The first ingredient concerns getting rid of the small associated stars:
we use the level-$1$ inequality (Lemma \ref{lem:level-1 inequality})
to show that the $i$th contribution to (\ref{eq:associated stars}) is negligible
unless either $S_{A}\left(i\right)$ or $S_{A'}\left(i\right)$ is large.

The second ingredient provides a combinatorial analysis
of large associated (inverse) stars, which is motivated
by the heuristic that such stars should be essentially disjoint.
Writing $s_A := \sum_i s_A(i)$ and $s'_A = \sum_i s_A'(i)$,
we should therefore expect $s_A + s'_A$ to be essentially bounded by $\alpha$.
In terms of the rescaled star sizes
\[ v_i := \frac{s_A(i)}{s_A + s'_A}
\quad \text{ and } \quad
v_i' := \frac{s_A'(i)}{s_A + s'_A},\]
our goal is then essentially reduced to showing that
\[
\sum_{i=1}^{n}v_{i}^{2}+v_{i}v_{i}'+v_{i}'^{2}\le 1,
\]
and moreover that if the sum is close to $1$
then some $v_{i}$ or $v_{i}'$ is close to $1$,
which is equivalent to the required
approximation of $A$ by an associated star or inverse star.

\subsection{Relating significant negative contributions to associated stars}

We start the implementation of the above overview
by expressing the significant negative contributions
to $\left\langle \matB\matA,\matC\right\rangle$
(as described in Lemma \ref{lem:Being left with struc})
in terms of associated stars.

We write $\overrightarrow{1}$ for the all-ones vector in $\mathbb{R}^{n}$.
The following inequalities are immediate from the facts that
each coefficient $\matA_{-}$ is at least $-\aA$, and similarly for $B$ and $C$,
using the identity
$\left\langle
X\overrightarrow{1},
Y\overrightarrow{1}
\right\rangle
= \sum_{i,j,k} x_{ij} y_{ik}$
for any matrices $X=(x_{ij})$, $Y=(y_{ij})$.

\begin{lem}
\label{lem:Simple lower bound}
We have the following inequalities.
\[
\frac{1}{\left(n-1\right)^{2}}\left\langle \matB_{-}\Astruc,\Cstruc\right\rangle \ge-\frac{\beta}{\left(n-1\right)^{2}}\left\langle \Astruc^{t}\overrightarrow{1},\Cstruc^{t}\overrightarrow{1}\right\rangle =-\beta\sum_{i=1}^{n}s_{A}'\left(i\right)s_{C}'\left(i\right)
\]
\[
\frac{1}{\left(n-1\right)^{2}}\left\langle \Bstruc\Astruc,\matC_{-}\right\rangle \ge\frac{-\gamma}{\left(n-1\right)^{2}}\left\langle \Astruc\overrightarrow{1},\Bstruc^{t}\overrightarrow{1}\right\rangle =-\gamma\sum_{i=1}s_{A}\left(i\right)s_{B}'\left(i\right)
\]
\[
\frac{1}{\left(n-1\right)^{2}}\left\langle \Bstruc\matA_{-},\Cstruc\right\rangle \ge\frac{-\alpha}{\left(n-1\right)^{2}}\left\langle \Bstruc\overrightarrow{1},\Cstruc\overrightarrow{1}\right\rangle =-\alpha\sum_{i=1}s_{B}\left(i\right)s_{C}\left(i\right)
\]
\end{lem}

\subsection{Small associated stars have a small contribution}

We will now bound the terms appearing
in Lemma \ref{lem:Simple lower bound} by replacing the matrices
such as $\Astruc$ that represent the dictatorial structure of $A$
with certain matrices $\matA_{\star}$ that represent the star structure of $A$.

We let $\matA_{\star}$ be obtained from $\Astruc$ as follows. Each
row $e_{i}\Astruc$ of $\Astruc$ either corresponds to a large associated
star if the sum of its coefficients is $>\delta$ and otherwise it
corresponds to a small associated star. We let $\matA_{\star}$ be
obtained from $\Astruc$ by replacing all the rows that correspond
to small associated stars with $0$. We define $\matB_{\star}$
and $\matC_{\star}$ similarly.

Similarly, we also define matrices $\matA'_{\star},\matB_{\star}',\matC_{\star}'$
that correspond to large associated inverse stars. We let $\matA_{\star}'$
be obtained from $\Astruc^{t}$ by replacing its rows that correspond
to small associated inverse stars with 0. We define $\matB'_{\star}$
and $\matC'_{\star}$ similarly.

The following lemma will be used to replace the terms
$\left\langle \Astruc^{t}\overrightarrow{1},\Cstruc^{t}\overrightarrow{1}\right\rangle ,$
$\left\langle \Astruc\overrightarrow{1},\Bstruc^{t}\overrightarrow{1}\right\rangle ,$
and $\left\langle \Bstruc\overrightarrow{1},\Cstruc\overrightarrow{1}\right\rangle $
by the corresponding terms $\left\langle \matA'_{\star}\overrightarrow{1},\matC{}_{\star}'\overrightarrow{1}\right\rangle ,$
$\left\langle \matA_{\star}\overrightarrow{1},\matB_{\star}'\overrightarrow{1}\right\rangle ,$
and $\left\langle \matB_{\star}\overrightarrow{1},\matC_{\star}\overrightarrow{1}\right\rangle$.
The applicability of the lemma to these terms
will follow from the level-$1$ inequality (Lemma \ref{lem:level-1 inequality}).

\begin{lem}
\label{lem:reducing to star structure}
Let $M=\left(m_{ij}\right),N=\left(n_{ij}\right)$
be matrices with entries in $[n^{-2},1]$.
Let $\eta_{1},\eta_{2}>0,$ and suppose for each $i$ that
\[
\|\mathrm{M}_{\left(2^{-i},2^{1-i}\right]}\|_{2}^{2}\le\eta_{1}2^{-i}\log^{O\left(1\right)}n,
\]
\[
\|\mathrm{N}_{\left(2^{-i},2^{1-i}\right]}\|_{2}^{2}\le\eta_{2}2^{-i}\log^{O\left(1\right)}n.
\]
Define matrices $M_{\star} = (m^\star_{ij})$ and $N_{\star} = (n^\star_{ij})$ by
\[
m^\star_{ij} = m_{ij} 1_{\sum_{j'} m_{ij'} >\delta\eta_{1}}
\quad \text{ and } \quad
n^\star_{ij} = n_{ij} 1_{\sum_{j'} n_{ij'} >\delta\eta_{2}}.
\]
Then we have
\begin{align*}
\left|\left\langle M\overrightarrow{1},N\overrightarrow{1}\right\rangle -\left\langle M_{\star}\overrightarrow{1},N_{\star}\overrightarrow{1}\right\rangle \right| & \le\delta\eta_{1}\eta_{2}\log^{O\left(1\right)}n\\
 & +\delta\eta_{1}\|N_{\star}\|_{1}+\delta\eta_{2}\|M_{\star}\|_{1}.
\end{align*}
\end{lem}

\begin{proof}
We have
\[
\left|\left\langle \left(M-M_{\star}\right)\overrightarrow{1},N_{\star}\overrightarrow{1}\right\rangle \right|
 \le\|\left(M-M_{\star}\right)\overrightarrow{1}\|_{\infty}\|N_{\star}\overrightarrow{1}\|_{1}
 \le\delta\eta_{1}\|N_{\star}\|_{1},
\]
and similarly
\[
\left|\left\langle M_{\star}\overrightarrow{1},N-N_{\star}\overrightarrow{1}\right\rangle \right|
\le\delta\eta_{2}\|M_{\star}\|_{1}.
\]
 Let $\tilde{M}=M-M_{\star}$ and $\tilde{N}=N-N_{\star}$. It remains
to show that
\[
\left\langle \tilde{M}\overrightarrow{1},\tilde{N}\overrightarrow{1}\right\rangle \le\delta\eta_{1}\eta_{2}\log^{O\left(1\right)}n.
\]
We will use the dyadic expansion
\[
\tilde{M}=\sum_{i=1}^{2\left\lceil \log n\right\rceil }\tilde{M}_{\left(2^{-i},2^{1-i}\right]}
\quad \text{ and } \quad
\tilde{N}=\sum_{i=1}^{2\left\lceil \log n\right\rceil }\tilde{N}_{\left(2^{-i},2^{1-i}\right]}.
\]
Each of the resulting terms will be bounded using the following claim.
\begin{claim}
Let $M',N'$ be matrices such that each row of $M'$
has at most $m'$ nonzero coefficients
and each row of $N'$ has at most $n'$
nonzero coefficients. Then
\[
\left|\left\langle M'\overrightarrow{1},N'\overrightarrow{1}\right\rangle \right|\le\sqrt{m'n'}\|M'\|_{2}\|N'\|_{2}.
\]
\end{claim}

\begin{proof}
This follows from Cauchy--Schwarz as
\begin{align*}
\left|\left\langle M'\overrightarrow{1},N'\overrightarrow{1}\right\rangle \right| & =\left|\sum_{i}\left\langle \left(M'\right)^{t}e_{i},\overrightarrow{1}\right\rangle \left\langle \left(N'\right)^{t}e_{i},\overrightarrow{1}\right\rangle \right|\\
 & \le\sum_{i}\sqrt{m'}\left\Vert \left(M'\right)^{t}e_{i}\right\Vert _{2}\sqrt{n'}\left\Vert \left(N'\right)^{t}e_{i}\right\Vert _{2}\\
 & \le\sqrt{m'n'}\left\Vert M'\right\Vert _{2}\left\Vert N'\right\Vert _{2}.
\end{align*}
\end{proof}
For each row $v$ of $\tilde{M}_{\left(2^{-i},2^{1-i}\right]}$,
all nonzero coefficients are $\ge 2^{-i}$
and $\left\langle v,1\right\rangle \le\delta\eta_{1}$,
so $v$ has $\le 2^{i}\delta\eta_{1}$ nonzero coefficients.
Similarly each row of $\tilde{N}_{\left(2^{-j},2^{1-j}\right]}$
has $\le 2^{-j}\delta\eta_{2}$ nonzero coefficients.
Applying the claim, we obtain
\begin{align*}
\left\langle \tilde{M}\overrightarrow{1},\tilde{N}\overrightarrow{1}\right\rangle = & \sum_{i=1}^{2\left\lceil \log n\right\rceil }\sum_{j=1}^{2\left\lceil \log n\right\rceil }\left\langle \tilde{M}_{\left(2^{-i},2^{1-i}\right]}\overrightarrow{1},\tilde{N}_{\left(2^{-j},2^{1-j}\right]}\overrightarrow{1}\right\rangle \\
\le & \sum_{i=1}^{2\left\lceil \log n\right\rceil }\sum_{j=1}^{2\left\lceil \log n\right\rceil }2^{\frac{i+j}{2}}\delta\sqrt{\eta_{1}\eta_{2}}\|\tilde{M}_{\left(2^{-i},2^{1-i}\right]}\|_{2}\|\tilde{N}_{\left(2^{-j},2^{1-j}\right]}\|_{2}\\
\le & 4\delta\eta_{1}\eta_{2}\log^{O\left(1\right)}n,
\end{align*}
where the final inequality uses the assumption of the lemma.
\end{proof}

\subsection{Combinatorial analysis of the star structure matrices}

In the next section we will use Lemma \ref{lem:reducing to star structure}
to complete our transition from dictatorial structure matrices $\matA_{\mathrm{struc}}$
to the star-structure matrices $\matA_{\star}$. To achieve this,
we first need to control the error terms that will arise
from applying Lemma \ref{lem:reducing to star structure},
namely the $L^1$-norms of the star structure matrices.

As discussed in the overview, this comes down to a combinatorial argument
showing that the corresponding large associated stars and inverse stars
are essentially disjoint, provided that $A,B,C$
(and so the parameters $\eps_A,\eps_B,\eps_C$) are sufficiently large.
This is captured by saying for each $E\in\left\{ A,B,C\right\}$
that $\sum\mu\left(E\cap S\right) \approx \mu\left(E\right)$,
where the sum is over all large associated stars and inverse stars of $E$.
\begin{lem}
\label{lem:associated stars are nearly disjoint}
Let $\epsilon>0,\delta>0$ and $E\subseteq S_{n}$.
Let $\mathcal{S}$ be a collection of stars
and inverse stars such that for each $S\in\mathcal{S}$
we have $\delta\mu\left(E\right)\le\mu\left(E\cap S\right)$
and $\frac{\epsilon}{2}\mu\left(S\right)\le\mu\left(E\cap S\right)$.
Suppose also that $
\frac{100}{\delta^{2}n}\le\mu\left(E\right)\le\frac{\delta^{2}\epsilon^{2}}{100}.$
Then we have
\[
\sum_{S\in\mathcal{S}}\mu\left(E\cap S\right)
\le\mu\left(E\right)+\frac{20\mu\left(E\right)^{2}}{\delta^{2}\epsilon^{2}}+\frac{4}{\delta^{2}n}.
\]
\end{lem}

\begin{proof}
Fix $\mathcal{S}' \sub \mathcal{S}$
with $|\mathcal{S}'| = \min\{\frac{2}{\delta},|\mathcal{S}|\}$.
By Inclusion--Exclusion we have
\begin{align*}
\mu\left(E\right) & \ge\sum_{S\in\mathcal{S}'}\mu\left(E\cap S\right)-\sum_{S_{1},S_{2}\in\mathcal{S}'}\mu\left(S_{1}\cap S_{2}\right).
\end{align*}
For any distinct $S_{1},S_{2}\in\mathcal{S}$ we have
\[
\mu\left(S_{1}\cap S_{2}\right)
\le\left(\frac{n}{n-1}\right)\mu\left(S_{1}\right)\mu\left(S_{2}\right)+\frac{1}{n}
\le\frac{5}{\epsilon^{2}}\mu\left(E\right)^{2}+\frac{1}{n}.
\]
Rearranging, we obtain
\begin{equation}
\sum_{S\in\mathcal{S}'}\mu\left(E\cap S\right)\le\mu\left(E\right)+\frac{20}{\delta^{2}\epsilon^{2}}\mu\left(E\right)^{2}+\frac{4}{\delta^{2}n},\label{eq:large associated stars}
\end{equation}
 which also yields
\[
\sum_{S\in\mathcal{S}'}\mu\left(E\cap S\right)<2\mu\left(E\right).
\]
To complete the proof we show that $\mathcal{S}'=\mathcal{S}.$ This
will follow once we show that $\left|\mathcal{S}'\right|<\frac{2}{\delta}.$
In fact we have
\[
\left|\mathcal{S}'\right|\delta\mu\left(E\right)\le\sum_{S\in\mathcal{S}'}\mu\left(E\cap S\right)<2\mu\left(E\right),
\]
 and hence $\left|\mathcal{S}'\right|<\frac{2}{\delta}.$
\end{proof}
Lemma \ref{lem:associated stars are nearly disjoint} can be restated
in terms of the star-structure matrices. It translates to the following
upper bound on their $1$-norms.
\begin{lem}
\label{lem:l1-bound} Suppose that
\[
\frac{1000}{\delta^{5}n^{2}}\le\alpha\min\left(\beta,\gamma\right)^{2}.
\]
Then we have
\[
\frac{1}{n-1}\left(\|\matA_{\star}\|_{1}+\|\matA'_{\star}\|_{1}\right)\le\alpha\left(1+\delta\right).
\]
Analogous statements hold for $B$ and $C$.
\end{lem}

\begin{proof}
Let $\mathcal{S}_{A}$ be the set of large associated stars for $A.$
For each $S\in\mathcal{S}_{A}$ we have $\mu\left(A\cap S\right)\ge\delta\mu\left(A\right)$
by definition. On the other hand, for each dictator $1_{i\to j}$
contained in $S$ we have
\[
\epsilon_A <a_{ij}=\frac{n-1}{n}\left(\mu\left(A_{i\to j}\right)-\alpha\right),
\]
which gives $\mu\left(A_{i\to j}\right)\ge \epsilon_A$,
and so $\mu\left(A\cap S\right)\ge \epsilon_A \mu\left(S\right)$.
Hence we can apply Lemma \ref{lem:associated stars are nearly disjoint}
with $\epsilon=\epsilon_{A}$, which gives
\[
\frac{1}{n-1}\left(\|\matA_{\star}\|_{1}+\|\matA'_{\star}\|_{1}\right)
\le \sum_{S\in\mathcal{S}_{A}}\mu\left(A\cap S\right)
\le \alpha+\frac{20\alpha^{2}}{\delta^{2}\epsilon_{A}^{2}}+\frac{4}{\delta^{2}n}.
\]
Substituting $\eps_A = n\dD \alpha\min(\beta,\gamma)$ and using
$\frac{1000}{\delta^{5}n^{2}}\le\alpha\min\left(\beta,\gamma\right)^{2}$
gives $\eps_A^2/\alpha \ge 1000 \dD^{-3}$, so the lemma follows.
\end{proof}

\subsection{Reducing to large associated stars}

The following lemma combines everything we proved so far.
It reduces us to upper bounding the star structure inner products
$\left\langle \matA'_{\star}\overrightarrow{1},\matC'_{\star}\overrightarrow{1}\right\rangle ,$
$\left\langle \matA_{\star}\overrightarrow{1},\matB'_{\star}\overrightarrow{1}\right\rangle $
and $\left\langle \matB_{\star}\overrightarrow{1},\matC_{\star}\overrightarrow{1}\right\rangle .$
\begin{lem}
\label{lem:wrapping up}
Suppose that $\left(A,B,C\right)$ are as in
Proposition \ref{prop:Explanation for cross product-freeness}.
Then we have
\begin{align*}
&\Pr_{\sigma,\tau\sim S_{n}}\left[\sigma\in A,\tau\in B,\sigma\tau\in C\right]
 \ge 2\alpha\beta\gamma(1-\dD^{1/4}) \\
 & \qquad - \frac{2}{(n-1)^2} \left| \beta\left\langle \matA'_{\star}\overrightarrow{1},\matC'_{\star}\overrightarrow{1}\right\rangle +\gamma\left\langle \matA_{\star}\overrightarrow{1},\matB'_{\star}\overrightarrow{1}\right\rangle +\alpha\left\langle \matB_{\star}\overrightarrow{1},\matC_{\star}\overrightarrow{1}\right\rangle \right|.
\end{align*}
\end{lem}

\begin{proof}
By Proposition \ref{Prop:Most weight is on the first two levels}
(reducing to the linear term) we have
\begin{align*}
\Pr_{\sigma,\tau\sim S_{n}}\left[\sigma\in A,\tau\in B,\sigma\tau\in C\right]
& \ge2\alpha\beta\gamma+\frac{2}{\left(n-1\right)^{2}}\left\langle \matB\matA,\matC\right\rangle -O\left(\frac{\sqrt{\alpha\beta\gamma}}{n^{2}}\right).
\end{align*}
Plugging in Lemma \ref{lem:Being left with struc}
(reducing to the significant negative terms)
we obtain
\begin{align*}
\left\langle \matB\matA,\matC\right\rangle
& \ge\left\langle \matB_{-}\Astruc,\Cstruc\right\rangle
+\left\langle \Bstruc\matA_{-},\Cstruc\right\rangle
+\left\langle \Bstruc\Astruc,\matC_{-}\right\rangle \\
 & -\sqrt{\delta}n^{2} (\log n)^{O\left(1\right)} \aA\bB\gG.
\end{align*}
By Lemma \ref{lem:Simple lower bound}
(lower bounding the negative entries)
we then obtain
\begin{align*}
\left\langle \matB\matA,\matC\right\rangle  & \ge-\beta\left\langle \Astruc^{t}\overrightarrow{1},\Cstruc^{t}\overrightarrow{1}\right\rangle -\gamma\left\langle \Astruc\overrightarrow{1},\Bstruc^{t}\overrightarrow{1}\right\rangle -\alpha\left\langle \Bstruc\overrightarrow{1},\Cstruc\overrightarrow{1}\right\rangle \\
 & -\sqrt{\delta}n^{2} (\log n)^{O\left(1\right)} \aA\bB\gG.
\end{align*}
It remains to approximate the above matrix inner products
by the corresponding terms involving star structure matrices.
The calculations for the three terms are analogous,
so we only show the details for the first term
$-\beta\left\langle \Astruc^{t}\overrightarrow{1},\Cstruc^{t}  \overrightarrow{1}\right\rangle $.
We will apply Lemma \ref{lem:reducing to star structure}
to $M = \Astruc^{t}$ and $N = \Cstruc^{t}$.
The level-$1$ inequality (Lemma \ref{lem:level-1 inequality})
shows that its hypotheses are satisfied with
$\eta_1 = \aA n$ and $\eta_2 = \gG n$.
Thus we obtain
\[ \left|\left\langle \Astruc^{t} \overrightarrow{1}, \Cstruc^{t} \overrightarrow{1}\right\rangle
-\left\langle \matA'_{\star}\overrightarrow{1}, \matC'_{\star}\overrightarrow{1}\right\rangle \right|
 \le\delta\aA\gG n^2 (\log n)^{O\left(1\right)}
  +\delta \aA n \|\matC'_{\star}\|_{1} + \delta \gG n \|\matA'_{\star}\|_{1}.\]
By Lemma \ref{lem:l1-bound} we have $\|\matA'_{\star}\|_{1} \le \aA(1+\dD)n$
and $\|\matC'_{\star}\|_{1} \le \gG(1+\dD)n$.
Recalling that $\left\langle \matB\matA,\matC\right\rangle$
is multiplied by $\frac{2}{\left(n-1\right)^{2}}$ in the main calculation,
we see that replacing
$-\beta\left\langle \Astruc^{t}\overrightarrow{1},\Cstruc^{t}\overrightarrow{1}\right\rangle $
by $-\beta\left\langle \matA'_\star \overrightarrow{1}, \matC'_\star \overrightarrow{1}\right\rangle $
incurs an error term $\le \dD\aA\bB\gG (\log n)^{O\left(1\right)}$.
Similar reasoning applies to the other terms, so the lemma follows.
\end{proof}

\subsection{Product-free triples are somewhat explained by stars}

We have now prepared the two main ingredients
described in the proof overview earlier in the section:
we have shown that the contributions
from small associated stars are negligible
and reduced the analysis of large associated stars
to bounding the star structure terms that appear
in Lemma \ref{lem:wrapping up}.
Our final lemma in preparation for the proof of
Proposition \ref{prop:Explanation for cross product-freeness},
shows that the star structure terms are small
except where they have a common row with a large sum,
which corresponds to the associated stars with common centre required to prove
Proposition \ref{prop:Explanation for cross product-freeness}.

\begin{lem}
\label{lem:finding a very large associated star } Let $\eta_{1},\eta_{2}>0$.
Let $M_{\star},N_{\star}$ be matrices with nonnegative coefficients
and suppose that there is no coordinate $i$
with both $(M_{\star}\overrightarrow{1})_i > \frac{\eta_{1}}{100}$
and $(N_{\star}\overrightarrow{1})_i > \frac{\eta_{2}}{100}.$ Then
\[
\left\langle M_{\star}\overrightarrow{1},N_{\star}\overrightarrow{1}\right\rangle \le\frac{\eta_{2}}{100}\|M_{\star}\|_{1}+\frac{\eta_{1}}{100}\|N_{\star}\|_{1}.
\]
\end{lem}

\begin{proof}
The terms of $\left\langle M_{\star}\overrightarrow{1},N_{\star}\overrightarrow{1}\right\rangle $
that correspond to a coordinate in which $M_{\star}\overrightarrow{1}$
is $\le\frac{\eta_{1}}{100}$ sum up to at most $\frac{\eta_{1}}{100}\|N_{\star}\overrightarrow{1}\|_{1}.$
In the rest of the terms the corresponding coordinate of $N_{\star}\overrightarrow{1}$
is at most $\eta_{2}/100.$ We therefore have
\begin{align*}
\left\langle M_{\star}\overrightarrow{1},N_{\star}\overrightarrow{1}\right\rangle  & \le\frac{\eta_{1}}{100}\|N_{\star}\overrightarrow{1}\|_{1}+\frac{\eta_{2}}{100}\|M^{\star}1\|_{1}\\
 & =\frac{\eta_{1}}{100}\|N_{\star}\|_{1}+\frac{\eta_{2}}{100}\|M^{\star}\|_{1}.
\end{align*}
\end{proof}
We are now ready to prove Proposition \ref{prop:Explanation for cross product-freeness}.
\begin{proof}[Proof of Proposition \ref{prop:Explanation for cross product-freeness}]
 By Lemma \ref{lem:wrapping up} we have
\[
\alpha\beta\gamma \le\frac{2}{n^{2}} \left(
\beta\left\langle \matA'_{\star}\overrightarrow{1},\matC'_{\star}\overrightarrow{1}\right\rangle
+\gamma\left\langle \matA_{\star}\overrightarrow{1},\matB'_{\star}\overrightarrow{1}\right\rangle
+\alpha\left\langle \matB_{\star}\overrightarrow{1},\matC_{\star}\overrightarrow{1}\right\rangle
.\right).
\]
If there is no coordinate $i$ in which
$(\matA'_{\star}\overrightarrow{1})_i > \aA n/100$ and
$(\matC'_{\star}\overrightarrow{1})_i > \gG n/100$
then Lemma \ref{lem:finding a very large associated star } and Lemma \ref{lem:l1-bound} bound
$\left\langle \matA'_{\star}\overrightarrow{1},\matC'_{\star}\overrightarrow{1}\right\rangle$
by $\| \matA'_{\star} \|_1 \gG n/100 + \| \matC'_{\star} \|_1 \aA n/100 < \frac{1}{40} \aA \gG n^2$.
Similar calculations apply to the other two terms.
However, these bounds cannot all hold, as then the inequality above
would give $\aA \bB \gG < 3 \cdot \frac{2}{40} \aA \bB \gG$, which is a contradiction.
Thus we have the required coordinate $i$ for one of the terms,
e.g.\ for the third term this would give
$(\matB_{\star}\overrightarrow{1})_i > \bB n/100$ and
$(\matC_{\star}\overrightarrow{1})_i > \gG n/100$,
so $\mu_B(S_B(i))>1/100$ and $\mu_C(S_C(i))>1/100$,
which corresponds to the stars described in the Proposition.
\end{proof}

\subsection{A dense product-free set is explained by a single star}

Now we will specialise the analysis of star structure terms
from product-free triples to a single product-free set,
where we will obtain the stronger structural conclusion
that a single star explains the lack of products.
We require the following lemma that bounds the star structure terms
under the assumption that no single star accounts for almost all
of the star structure matrices.

\begin{lem}
\label{lem:Product-freeness is explained by a single star}
Let $\zeta \in [0,1/2]$ and suppose that
\[
\max\left(\|\matA_{\star}\overrightarrow{1}\|_{\infty},\|\matA'_{\star}\overrightarrow{1}\|_{\infty}\right)\le\left(1-\zeta\right)\left(\|\matA_{\star}\|_{1}+\|\matA'_{\star}\|_{1}\right).
\]
 Then
\[
\left\langle \matA_{\star}\overrightarrow{1},\matA_{\star}\overrightarrow{1}\right\rangle +\left\langle \matA'_{\star}\overrightarrow{1},\matA'_{\star}\overrightarrow{1}\right\rangle +\left\langle \matA'_{\star}\overrightarrow{1},\matA{}_{\star}\overrightarrow{1}\right\rangle   \le
\left(1-\zeta/2\right)\left(\|\matA_{\star}\|_{1}+\|\matA'_{\star}\|_{1}\right)^{2}.
\]
\end{lem}

\begin{proof}
The lemma is immediate from the following claim applied to
$v=\frac{\matA_{\star}\overrightarrow{1}}{\|\matA_{\star}\|_{1}+\|\matA'_{\star}\|_{1}}$
and $u=\frac{\matA'_{\star}\overrightarrow{1}}{\|\matA_{\star}\|_{1}+\|\matA'_{\star}\|_{1}}.$
\begin{claim}
\label{claim:inequality for stars}
Let $v,u \in [0,1-\zeta]^n$ with $\|v\|_{1}+\|u\|_{1}=1$. Then
\[
\|v\|_{2}^{2}+\|u\|_{2}^{2}+\left\langle v,u\right\rangle \le 1-\zeta(1-\zeta).
\]
\end{claim}

\begin{proof}
Suppose that $v_{1}+u_{1}\ge v_{i}+u_{i}$ for all other $i.$ Then
we have
\begin{align*}
\|v\|_{2}^{2}+\|u\|_{2}^{2}+\left\langle v,u\right\rangle  & =\sum_{i=1}^{n}\left(\left(v_{i}+u_{i}\right)^{2}-v_{i}u_{i}\right)\\
 & \le\sum_{i=1}^{n}\left(v_{1}+u_{1}\right)\left(v_{i}+u_{i}\right)-v_{1}u_{1}\\
 & =v_{1}+u_{1}-v_{1}u_{1}.
\end{align*}
The function $f(u_1,v_1) :=v_{1}+u_{1}-v_{1}u_{1}$ is increasing in both coordinates,
so is maximised when $u_1+v_1=1$. Now $g(v_1) := f(1-v_1,v_1)=1-v_1(1-v_1)$
is a convex function of $v_1$, so is maximised at the boundary of its domain
$[\zeta,1-\zeta]$. The claim follows.
\end{proof}
\end{proof}
Now we can show for any moderately dense product-free set $A \sub A_n$
that some associated star or inverse star explains $A$
up to a factor $1+O(\dD^{1/4})$.

\begin{lem}
\label{lem:Explanation by associated stars}
Suppose that $A\subseteq A_{n}$ is product-free
with $\mu\left(A\right)\ge \dD^{-2} n^{-2/3}$.
Then $\max_{i \in [n]} \max\{s_A(i), s'_A(i)\} = (1+O(\dD^{1/4}))\aA$.
\end{lem}

\begin{proof}
We define $\zeta' \in [0,1]$ by
\[ \max\left(\|\matA_{\star}\overrightarrow{1}\|_{\infty},\|\matA'_{\star}\overrightarrow{1}\|_{\infty}\right)
= (1-\zeta')( \|\matA_{\star}\|_{1}+\|\matA'_{\star}\|_{1} ) \]
and let $\zeta = \min(\zeta',1/2)$.
Applying Lemma \ref{lem:wrapping up} with $A=B=C$,
then Lemma \ref{lem:Product-freeness is explained by a single star},
and then Lemma \ref{lem:l1-bound}, we obtain
\begin{align*}
(n-1)^2 \alpha^{3} \left(1-O\left(delta^{1/4}\right)\right)
& \le \alpha\left(\left\langle \matA'_{\star}\overrightarrow{1},\matA'_{\star}\overrightarrow{1}\right\rangle +\left\langle \matA_{\star}\overrightarrow{1},\matA'_{\star}\overrightarrow{1}\right\rangle +\left\langle \matA_{\star}\overrightarrow{1},\matA_{\star}\overrightarrow{1}\right\rangle \right) \\
& \le \alpha\left(\|\matA_{\star}\|_{1}+\|\matA'_{\star}\|_{1}\right)^{2}\left(1-\zeta/2\right) \\
& \le n^2 \alpha^3 \left(1-\zeta/2\right)\left(1+O\left(\delta\right)\right).
\end{align*}
Thus $\zeta=O(\delta^{1/4})$, so $\zeta'=\zeta$,
and $\|\matA_{\star}\|_{1}+\|\matA'_{\star}\|_{1} = (1+O(\delta^{1/4})) \aA n$,
so by definition of $\zeta'$ we have
\[ (n-1) \max_{i \in [n]} \max\{s_A(i), s'_A(i)\} =
\max\left(\|\matA'_{\star}\overrightarrow{1}\|_{\infty},\|\matA'_{\star}\overrightarrow{1}\|\right)
= (1+O(\delta^{1/4})) \aA n.\]
This completes the proof.
\end{proof}
\subsection{A dense product-free set is close to a single star}

Now we will refine the star structure obtained in the previous subsection
to prove Proposition \ref{prop:Explanation for product-freeness},
which shows that any moderately dense product-free set
is closely approximated by a single star.
Our idea is to use the (inverse) star $S$ provided above
with the fact that $\left(A,A,A\setminus S\right)$ is a product-free triple
to which we can apply Proposition \ref{prop:Explanation for cross product-freeness}
to deduce that $A\setminus S$ is small.

\begin{proof}[Proof of Proposition \ref{prop:Explanation for product-freeness}]
By Lemma \ref{lem:Explanation by associated stars}
there is a large associated (inverse) star $S_{1}$
such that $\mu\left(A\setminus S_{1}\right)<O(\dD^{1/4}) \mu\left(A\right).$
Without loss of generality $S_{1}=1_{1\to I_{1}}.$ Let
\[
I=\left\{ i:\,\mu\left(A_{1\to i}\right)\ge n^{-\frac{1}{3}}\right\} ,
\]
 let $S=1_{1\to I}$, and let $C=A\setminus S.$

We assert that $\mu\left(C\right)\le\delta^{-2}n^{-2/3}.$
Suppose otherwise. Then the triple $\left(A,A,C\right)$ is product-free,
and we can apply Proposition \ref{prop:Explanation for cross product-freeness}
to deduce that its conclusion holds for some triple equivalent to $(A,A,C)$.
In principle, the possibilities are for some $i$ that
\begin{enumerate}
\item $\mu_A(S_A(i))>1/100$ and $\mu_A(S'_A(i))>1/100$, or
\item $\mu_A(S'_A(i))>1/100$ and $\mu_C(S'_C(i))>1/100$, or
\item $\mu_A(S_A(i))>1/100$ and $\mu_C(S_C(i))>1/100$.
\end{enumerate}
However, $\mu_A(S_A(1))$ is so large that it precludes
the existence of any other associated star or inverse star $S$
with $\mu_A(S)>1/100$, so the only possibility is that (3) holds with $i=1$.
However, by definition of $C$ we have
$\mu\left(C_{1\to j}\right)\le n^{-\frac{1}{3}}$ for all $i$,
but as $S_C(1) \ne \es$ we can choose $j$ with
\[ \mu(C_{1 \to j}) > c_{1j} > \eps_C
=n\delta\mu\left(C\right)\mu\left(A\right)>n^{-\frac{1}{3}}.
\]
This contradiction completes the proof.
\end{proof}

\subsection{Product-free triples when one set has no large associated stars}

We have now completed the proofs for the main goals of the section.
For future reference we will conclude the section by
proving a star-structure theorem for product-free triples $\left(A,B,C\right)$
where $A$ has no large associated stars and inverse stars.
The rationale is that after reordering we expect $B$ and $C$
to look like stars and $A$ to look like $1_{I\to\overline{J}}.$
We therefore expect $A$ to have no large associated stars.
Under this assumption, we strengthen the conclusion of
Proposition \ref{prop:Explanation for cross product-freeness}
(product-free triples are somewhat explained by stars)
to the same level of accuracy that we achieved for a single product-free set:
we show that $B$ and $C$ are each explained up to a factor $1+O(\dD^{1/4})$
by a single star with the same centre.

\begin{lem}
\label{lem:product-freeness in case A has no large associated stars}
Suppose that $\left(A,B,C\right)$ are as in
Proposition \ref{prop:Explanation for cross product-freeness}
and that $A$ has no associated stars and inverse stars.
Then there exists $i\in\left[n\right]$ such that
$\mu\left(B\setminus S_{B}\left(i\right)\right)\le O(\dD^{1/4})\beta$
and $\mu\left(C\setminus S_{C}\left(i\right)\right)\le O(\dD^{1/4})\gamma.$
\end{lem}

\begin{proof}
The proof is similar to that of Lemma \ref{lem:Explanation by associated stars}.
By assumption we have $\matA_{\star}$ and $\matA'_{\star}=0$,
so Lemma \ref{lem:wrapping up} reduces to
\[
(n-1)^2 \alpha\beta\gamma\left(1-O\left(\dD^{1/4}\right)\right)
=\alpha\left\langle \matB_{\star}\overrightarrow{1},\matC_{\star}\overrightarrow{1}\right\rangle .
\]
Let $v=\frac{\matB_{\star}\overrightarrow{1}}{\|\matB_{\star}\|_{1}}$
and $u=\frac{C_{\star}\overrightarrow{1}}{\|\matC_{\star}\|_{1}}$.
Then $\left\langle v,u\right\rangle > 1-O(\dD^{1/4})$,
as otherwise Lemma \ref{lem:l1-bound} would give
\[
(n-1)^2 \alpha\beta\gamma
< (1-O(\dD^{1/4})) \alpha  \|\matB_{\star}\|_{1} \|\matC_{\star}\|_{1}
< (1-O(\dD^{1/4})) \aA (1+\dD) \bB n (1+\dD) \gG n,
\]
which is a contradiction.
Now in the place of Claim \ref{claim:inequality for stars} we must
show that there is some coordinate $i$ with both $v_{i}$ and $u_{i}$
at least $1-O\left(\dD^{1/4}\right)$. As
\[ \left\langle v,u\right\rangle \le\|v\|_{\infty}\|u\|_{1}=\|v\|_{\infty} \]
we have $\|v\|_{\infty}\ge 1-O\left(\dD^{1/4}\right)$
and similarly $\|u\|_{\infty}\ge 1-O\left(\dD^{1/4}\right)$,
so each of $u$ and $v$ has one coordinate equal to $1-O\left(\dD^{1/4}\right)$
with the others all $O\left(\dD^{1/4}\right)$.
It remains to note that $\|v\|_{\infty}$ and $\|u\|_{\infty}$
must be achieved at the same coordinate, as otherwise
we would have
\[
\left\langle v,u\right\rangle \le O\left(\dD^{1/4}\right)\left(\|v\|_{1}+\|u\|_{1}\right)
= O\left(\dD^{1/4}\right),
\]
which contradicts $\left\langle v,u\right\rangle > 1-O(\dD^{1/4})$.
\end{proof}

\section{Bootstrapping} \label{sec:boot}

In this section we will use the star structure
established in the previous section to prove our main results,
which give exact extremal results and strong stability results
for product-free sets in $A_n$. The proofs will use the large
restrictions provided by the star structure to deduce that
other restrictions must be much smaller. This will allow us to progressively
tighten our approximate structure until it becomes exact.

\subsection{Product-free restrictions}

As discussed in subsection \ref{subsec:restrict},
if some product-free $A \sub A_n$ is well approximated by a star $1_{x\to I}$
then for each $i,i'\in I$ we will see that $A$ has small density
in $\mathcal{D}_{i\to i'}$ by inspecting the triple
$(A_{i\to i'},A_{x\to i},A_{x\to i'})$
and factoring out the corresponding dictators.
This is formalised by the following lemma.

\begin{lem}
\label{lem:eberhard for restrictions} Let $\epsilon>\frac{1}{\delta\sqrt{n}},$
and let $\left(A,B,C\right)$ be product-free.
Suppose that $\mu\left(B_{i\to j}\right) \ge \eps$
and $\mu\left(C_{i\to k}\right)\ge\epsilon.$
Then
\[
\mu\left(A_{j\to k}\right)\le\frac{\log^{O\left(1\right)}\left(1/\epsilon\right)}{\epsilon n}.
\]
\end{lem}

\begin{proof}
We apply the following transformation that preserves products:
let $B'=\left(jn\right)B_{i\to j}\left(ni\right),$ $A'=\left(nk\right)A_{j\to k}\left(nj\right)$
and $C'=\left(nk\right)C_{i\to k}\left(ni\right).$ Then the equation
$ab=c$ in $\left(A',B',C'\right)$ is equivalent to the corresponding
equation inside $\left(A_{j\to k},B_{i\to j},C_{i\to k}\right).$
As $\left(A',B',C'\right)$ can be viewed as product-free subsets
of $A_{n-1}$ the lemma follows from Corollary \ref{cor:unbalanced eberhard}.
\end{proof}

The following lemma shows that if $\left(A,B,C\right)$ are product-free
and $B$, $C$ are dense in stars $1_{x\to I}$, $1_{x\to J},$
then $A$ is sparse outside of
$1_{I\to\overline{J}} := \{ \sS: \sS(I) \cap J = \es \}$.
\begin{lem}
\label{lem:bootstrapping lemma 1} Let $\epsilon>\frac{1}{\delta\sqrt{n}}$
and let $\left(A,B,C\right)$ be a product-free triple. Let $x\in\left[n\right]$
and let $I,J\subseteq\left[n\right]$. Suppose that for each $i\in I$
and each $j\in J$ we have $\mu\left(B_{x\to i}\right),\mu\left(C_{x\to j}\right)\ge\epsilon$.
Suppose further that $\alpha\ge\frac{1}{\delta\epsilon n}.$
Then $\left|I\right|\left|J\right|\le 10 n\log n$
and
\[
\mu\left(A\setminus1_{I\to\overline{J}}\right)\le\frac{\left|I\right|\left|J\right|\log^{O\left(1\right)}\left(1/\epsilon\right)}{\epsilon n^{2}}.
\]
\end{lem}

\begin{proof}
By a union bound and Lemma \ref{lem:eberhard for restrictions} we
have
\begin{equation}
\mu\left(A\setminus1_{I\to\overline{J}}\right)\le\sum_{i\in I,j\in J}\frac{1}{n}\mu\left(A_{i\to j}\right)\le\frac{\left|I\right|\left|J\right|}{n}\frac{\log^{O\left(1\right)}(1/\eps)}{\epsilon n}.\label{eq:union bound}
\end{equation}
To complete the proof we show that $\left|I\right|\left|J\right|\le10n\log n.$
Suppose otherwise. Then by removing elements from either $I$ or $J$
we may assume that $\left|I\right|\left|J\right|\in\left(10n\log n,20n\log n\right).$
In which case we have
\begin{align*}
\mu\left(1_{I\to\overline{J}}\right) & =\left(1-\frac{\left|J\right|}{n}\right)\left(1-\frac{\left|J\right|}{n-1}\right)\cdots\left(1-\frac{\left|J\right|}{n-\left|I\right|}\right)\\
 & \le\left(1-\frac{\left|J\right|}{n}\right)^{\left|I\right|}\le e^{-\left|J\right|\left|I\right|/n}\le n^{-3},
\end{align*}
 Together with (\ref{eq:union bound}) this yields
\[
\alpha\le\frac{\log^{O\left(1\right)}n}{\epsilon n}+n^{-3},
\]
 which contradicts the hypothesis $\alpha\ge\frac{1}{\delta\epsilon n}.$
This shows that $\left|I\right|\left|J\right|\le10n\log n.$
\end{proof}

\subsection{Stability result for product-free sets}

We now prove the following stronger version of Theorem \ref{thm:99=000025 Stability}.
\begin{thm}
\label{Lem:stability result}
Suppose that $A$ is product-free
with $\mu\left(A\right)\ge\delta^{-2}n^{-2/3}$.
Then up to inversion there exist $x\in\left[n\right]$
and $I\subseteq\left[n\right]$ with $|I|^2 \le 10n\log n$ such that
$\mu\left(A_{x\to i}\right)\ge n^{-\frac{1}{3}}$ for each $i\in I$ and
\[
\mu\left(A\setminus F_{I}^{x}\right)\le O\left(\delta^{-2}\right)n^{-2/3}.
\]
\end{thm}

\begin{proof}
By Proposition \ref{prop:Explanation for product-freeness},
up to inversion there exist $x\in\left[n\right]$
and $I\subseteq\left[n\right]$ such that
for each $i\in I$ we have $\mu\left(A_{x\to i}\right)\ge n^{-1/3}$
and $\mu\left(A\setminus1_{x\to I}\right)\le O(\dD^{-2}) n^{-\frac{2}{3}}$.
Without loss of generality we assume this for $A$ rather than $A^{-1}$.
By Lemma \ref{lem:bootstrapping lemma 1} with $A=B=C$ and $\epsilon=n^{-\frac{1}{3}}$
we have $|I|^2 \le 10n\log n$ and
\[
\mu\left(A\setminus1_{I\to\overline{I}}\right)\le\frac{\log^{O\left(1\right)}n}{n^{2/3}}.
\]
 This shows that
\[
\mu\left(A\setminus F_{I}^{x}\right)\le\mu\left(A\setminus1_{x\to I}\right)+\mu\left(A\setminus1_{I\to\overline{I}}\right)\le O\left(\delta^{-2}\right)n^{-\frac{2}{3}}.
\]
\end{proof}

\subsection{Bootstrapping triples}

With slightly more work we are also able to prove the following
stability result for product-free triples that are somewhat sparse
(we replace a log factor by a log log factor).

\begin{thm}
\label{thm:Cross stability result}
There is an absolute constant $K$ such that if $n$ is sufficiently large
and $\left(A,B,C\right)$ is a product-free triple in $A_n$ with
\[
\min\left(\alpha\beta,\beta\gamma,\gamma\alpha\right)\ge\frac{(\log \log n)^K}{n}
\]
then up to equivalence there exist
$I,J\subseteq\left[n\right],x\in\left[n\right]$
with $\left|I\right|\left|J\right|\le 10n\log n$ such that
\[
\mu_{B}\left(1_{x\to I}\right),\mu_{C}\left(1_{x\to J}\right)\ge1-O\left(\dD^{1/4}\right)
\quad \text{ and } \quad
\mu\left(A\setminus1_{I\to\overline{J}}\right)\le \dD^{-2} n.
\]
In particular, we have
\[ \mu\left(A\setminus1_{I\to\overline{J}}\right) = o(\aA),
\quad \mu(B \sm 1_{x\to I}) = o(\bB),
\quad \mu(C \sm 1_{x\to J}) = o(\gG). \]
\end{thm}

The idea of the proof is to start with the weaker star structure
guaranteed from Proposition \ref{prop:Explanation for cross product-freeness}.
Lemma \ref{lem:bootstrapping lemma 1} will then easily imply a weaker
variant of the theorem with $1-O\left(\dD^{1/4}\right)$ replaced
by $\frac{1}{100}.$ This will allow us to show that actually $A$
has no large associated stars and inverse stars, so instead of Proposition
\ref{prop:Explanation for cross product-freeness} we may apply the
more suitable Lemma \ref{lem:product-freeness in case A has no large associated stars}.

\begin{proof}
We claim that $\alpha\ge\frac{1}{\delta^{100}n}.$
Indeed, otherwise our assumption implies
$\eps := \min(\beta,\gamma) > (\log \log n)^K \dD^{100}$,
and then Corollary \ref{cor:unbalanced eberhard}
gives $\aA \le \frac{\log^{O(1)}(1/\eps)}{\eps n}$,
so recalling $\dD^{-1} = \log^R n$ we have
$\min(\aA\beta,\aA\gamma) < \frac{(\log \log n)^{O(1)}}{n}$,
which contradicts our assumption for $K$ large enough.
Thus we have the claimed lower bound on $\aA$.

By Proposition \ref{prop:Explanation for cross product-freeness},
up to equivalence of the triple $(A,B,C)$ there exist associated stars
$S_{B}\left(x\right)=1_{x\to I},$ $S_{C}\left(x\right)=1_{x\to J}$,
such that $\mu_{B}\left(S_{B}\left(x\right)\right)>\frac{1}{100}$
and $\mu_{C}\left(S_{C}\left(x\right)\right)\ge\frac{1}{100}.$
Our assumptions implies $\min\left(\epsilon_{B},\epsilon_{C}\right)\ge\delta$,
so by Lemma \ref{lem:bootstrapping lemma 1} we have
$\left|I\right|\left|J\right|\le 10n\log n$ and
\[
\mu\left(A\setminus1_{I\to\overline{J}}\right)\le\frac{1}{\delta^{2}n}
\]

We now assert that $A$ has no large associated star or inverse star $S$.
Indeed, for such $S$, as $\mu\left(A\cap S\right)\ge\delta\mu\left(A\right)$
we would have
\[
\mu\left(S\right)\le\frac{\mu\left(A\right)}{\epsilon_{A}}\le\frac{\mu\left(A\right)}{\delta}\le\frac{\mu\left(A\cap S\right)}{\delta^{2}}.
\]
On the other hand,
\begin{align*}
\mu\left(A\cap S\right) & \le\mu\left(S\cap1_{I\to\overline{J}}\right)+\mu\left(A\setminus1_{I\to\overline{J}}\right)\\
 & \le2\mu\left(S\right)e^{-\frac{\left|I\right|\left|J\right|}{n}}+\frac{1}{\delta^2 n}.
\end{align*}
We have
\[  \frac{|I||J|}{n^2} \ge \mu(S_B(x)) \mu(S_C(x))
\ge \frac{\bB}{100} \frac{\gG}{100} > \frac{(\log \log n)^K}{10^4 n}, \]
so
\[ \mu\left(A\right) \le \dD^{-1} \mu\left(A\cap S\right)
 \le2\mu\left(S\right)\delta^{8}+\frac{1}{\delta^3 n}
 \le2\mu\left(A\right)\delta^{6}+\frac{1}{\delta^{3}n}. \]
This contradicts the bound $\mu(A)=\alpha\ge\frac{1}{\delta^{100}n}$,
so such $S$ cannot exist.

Thus we can apply Lemma \ref{lem:product-freeness in case A has no large associated stars}
to obtain the required stronger approximation of $B$ and $C$ by associated stars,
i.e.\ $\mu_{B}\left(S_{B}\left(x\right)\right)$
and $\mu_{C}\left(S_{C}\left(x\right)\right)$ are both $1-O\left(\dD^{1/4}\right)$.
\end{proof}

\subsection{Further bootstrapping of product-free sets}

Recall from Theorem \ref{Lem:stability result}
that if $A$ is a dense product-free set
then $\mu\left(A\setminus F_{I}^{x}\right)$
is small for some $x$ and $I$.
If $A$ is extremal this implies $\mu_{F_{I}^{x}}\left(A\right) \approx 1$.
Our goal in this subsection is to show for such $A$ that
\[
\mu\left(A\setminus F_{I}^{x}\right)\le\frac{1}{2}\mu\left(F_{I}^{x}\setminus A\right)
\]
and therefore any extremal product-free $A$ must be of the form $F_{I}^{x}.$
Our proof will also work for sets that are sufficiently close to extremal.

We require various lemmas that will be applied to certain restrictions of $A$.
The first considers a product-free triple $\left(A,B,C\right)$
(which will be restrictions of the original $A$)
and shows that if two sets are dense
in sets of the form $1_{I\to\overline{I}}$ then the third must be empty.

\begin{lem}
\label{lem:relative product-freensess with large set implies that the third is empty}
Let $I_{1},I_{2},J_{1},J_{2}\subseteq\left[n\right]$
with $n$ sufficiently large and
$\left|I_{1}\right|+\left|I_{2}\right|+\left|J_{1}\right|+\left|J_{2}\right|\le40\sqrt{n}.$
Suppose that $\left(A,B,C\right)$ is product-free.
If $\mu_{I_{1}\to\overline{J_{1}}}\left(C\right)$
and $\mu_{I_{2}\to\overline{J_{2}}}\left(B\right)$
are both at least $1-e^{-2000}$ then $A$ is empty.
Similar statements hold for all permutations of $ABC$.
\end{lem}

\begin{proof}
We only prove the first statement,
as the proofs for permutations of $ABC$ are similar.
Suppose on the contrary that there exists $\tau\in A$.
Let $\sigma\sim S_{n}$ be a random permutation.
We will derive a contraction by showing that
$\Pr\left[\sigma\in B,\tau\sigma\in C\right]>0.$
By assumption, we have
\[
\Pr\left[\sigma\in B,\tau\sigma\in C\right]\ge\Pr\left[\sigma\left(I_{1}\right)\subseteq\overline{J_{1}},\tau\sigma\left(I_{2}\right)\subseteq\overline{J_{2}}\right]-2e^{-2000}.
\]
On the other hand,
\begin{align*}
\Pr\left[\sigma\left(I_{1}\right)\subseteq\overline{J_{1}},\tau\sigma\left(I_{2}\right)\subseteq\overline{J_{2}}\right] & =\mu\left(1_{I_{1}\to\overline{J_{1}}}\cap1_{I_{2}\to\tau^{-1}\left(\overline{J_{2}}\right)}\right)\\
 & \ge\left(1-\frac{\left|I_{1}\right|+\left|I_{2}\right|+\left|J_{1}\right|+\left|J_{2}\right|}{n}\right)^{\left|I_{1}\right|+\left|I_{2}\right|}\\
 & >2e^{-2000},
\end{align*}
 provided that $n$ is sufficiently large. Indeed, the second inequality
follows from defining a random $\sigma$ on the indices
$i\in I_{1}\cup I_{2}$ one by one, noting that in each step
we have at least $n-|I_1|-|I_2|-|J_1|-|J_2|$ free options.
Thus $\Pr\left[\sigma\in B,\tau\sigma\in C\right]>0$,
contrary to our assumption that $(A,B,C)$ is product-free, so $A=\es$.
\end{proof}
We now show that if $\left(A,B,C\right)$ is a product-free triple,
$B$ contains almost all of $1_{I_{1}\to\overline{J_{1}}}$ and $A$
contains almost all of $F_{I}^{x}$ then $C$ is small.
\begin{lem}
\label{lem:large and  extremal means that the third is small}
Let $\zeta\in\left(0,e^{-2000}\right)$ and
let $\left(A,B,C\right)$ be a product-free triple
in $A_n$ with $n$ sufficiently large.
Suppose for some $I_{1},J_{1},I$ of size $\le10\sqrt{n}$
that $\mu_{I_{1}\to\overline{J_{1}}}\left(B\right)\ge1-\zeta$
and $\mu_{F_{I}^{x}}\left(A\right)\ge1-\frac{1}{2}e^{-2000}.$
Then
\[
\mu\left(C\right)\le\left(e^{2000}\zeta\right)^{\frac{\left|I\right|}{2}}.
\]
Moreover, if $\zeta<\frac{\left|I\right|}{2e^{2000}n}$ then $C$ is empty.
\end{lem}

\begin{proof}
Let
\[ I_{2} := \{ i \in I : \mu_{I\to\overline{I}}\left(A_{x\to i}\right)\ge1-e^{-2000} \}. \]
By definition of $F_{I}^{x}$, our assumption on $A$ and Markov's inequality
we have $\left|I_{2}\right|\ge\frac{\left|I\right|}{2}.$
Similarly, letting
\[ J := \{ j \in [n] :  \mu_{I_{1}\to\overline{J_{1}}}\left(B_{j\to x}\right)\ge1-e^{-2000} \}, \]
Markov's inequality applied to $B$ gives $|J| \ge\left(1-e^{2000}\zeta\right)n$.

For each $i\in I_{2}$ and $j\in J$ we can apply
Lemma \ref{lem:relative product-freensess with large set implies that the third is empty}
to the triple $\left(A_{x\to i},B_{j\to x},C_{j\to i}\right)$
to conclude that each such $C_{j\to i}$ is empty,
so $C\subseteq1_{J\to\overline{I_{2}}}.$ This shows that
\[
\mu\left(C\right)\le\mu\left(1_{J\to\overline{I_{2}}}\right)=\prod_{i=1}^{\left|I_{2}\right|}\frac{\left|\overline{J}\right|-i}{n-i}\le\left(1-\frac{\left|\overline{J}\right|}{n}\right)^{\left|I_{2}\right|}\le\left(e^{2000}\zeta\right)^{\left|I\right|/2}.
\]
Finally, if $\zeta<\frac{\left|I\right|}{2e^{2000}n}$
then $\left|\overline{I_{2}}\right|<\left|J\right|$,
and so $C \sub 1_{J\to\overline{I_{2}}}=\varnothing.$
\end{proof}

After the above preliminary lemmas, we now come to the main engine
of our bootstrapping approach, which shows that
if $\mu\left(A\setminus F_{I}^{x}\right)$ is somewhat small
then it is much smaller than $\mu\left(F_{I}^{x}\setminus A\right).$
\begin{lem}
\label{lem:bootstrapping lemma }
Let $A \sub A_n$ be product-free with $n$ sufficiently large.
Suppose that
\[ I := \{ i \in [n]: \mu\left(A_{1\to i}\right)\ge n^{-\frac{1}{3}} \}. \]
satisfies $\frac{\sqrt{n}}{10}<\left|I\right|\le10\sqrt{n}$.
Suppose $\mu_{F_{I}^{x}}\left(A\right)\ge1-\zeta$
with $\zeta\in\left(0,e^{-20000}\right)$.
Then
\begin{align*}
\mu\left(A\setminus F_{I}^{x}\right) & \le\zeta n^{-2/3}\log^{O\left(1\right)}n.
\end{align*}
 Moreover, if $\zeta\le\frac{1}{e^{8000}\sqrt{n}}$ then $A\subseteq F_{I}^{x}.$
\end{lem}

\begin{proof}
We define $I_2 \sub I_1 \sub I$ by
\[
I_1 = \{i \in I: \mu_{I\to\overline{I}}\left(A_{x\to i}\right)\ge 1-e^{-2000} \}
\quad \text{ and } \quad
I_2 = \{i \in I_1: \mu_{I\to\overline{I}}\left(A_{x\to i}\right)\ge 1-2\zeta \}.
\]
By Markov's inequality we have
$\left|I\setminus I_{1}\right|\le e^{2000}\zeta\left|I\right|$
and $\left|I_{2}\right|\ge\frac{1}{2}\left|I\right|$.

Let $J\subseteq\overline{I}$ consist of the $\left\lfloor 80\cdot e^{4000}\zeta\sqrt{n}\right\rfloor $
indices $i \in \overline{I}$ in which $\mu\left(A_{x\to i}\right)$ is largest.
We partition $\left(A\setminus F_{I}^{x}\right)$ into the following three bad events
and bound each of their measures separately.
\begin{enumerate}
\item Let $B_1 = \{ \sS \in A: \sS(x) \in I \text{ and } \sS(I) \cap I \ne \es \}$.
\item Let $B_2 = \{ \sS \in A: \sS(x) \in J\}$.
\item Let $B_3 = \{ \sS \in A: \sS(x) \in \overline{I \cup J} \}$.
\end{enumerate}
To prove the lemma, it suffices to show for each $i \in [3]$
that $\mu(B_i) < \zeta n^{-2/3}\log^{O\left(1\right)}n$,
and if moreover $\zeta\le\frac{1}{e^{8000}\sqrt{n}}$ then $B_i=\es$.

\subsection*{Upper bounding $\mu\left(B_{1}\right)$} 

By a union bound we have
\[
\mu\left(B_{1}\right)\le\sum_{i,j\in I}\frac{\mu\left(A_{i\to j}\right)}{n}.
\]
By definition of $I_1$, for each $i,j \in I_1$ we can apply
Lemma \ref{lem:relative product-freensess with large set implies that the third is empty}
to $\left(A_{x\to j},A_{i\to j},A_{x\to i}\right)$ to obtain $A_{i\to j}=\varnothing$.
For general $i,j\in I$ we may apply Lemma \ref{lem:eberhard for restrictions}
to obtain
\[
\mu\left(A_{i\to j}\right)\le n^{-\frac{2}{3}}\log^{O\left(1\right)}n.
\]
 Therefore,
\[
\mu\left(B_{1}\right)\le\sum_{i,j\in I}\frac{\mu\left(A_{i\to j}\right)}{n}\le2\left|I\setminus I_{1}\right|\left|I\right|n^{-5/3}\log^{O\left(1\right)}n\le\zeta n^{-2/3}\log^{O\left(1\right)}n.
\]
Moreover, if $\zeta\le\frac{1}{e^{8000}\sqrt{n}}$
then $|I \sm I_1| \le e^{2000}\zeta|I| < e^{-6000}|I|/\sqrt{n}<1$,
i.e.\ $I_{1}=I$ and so $B_1 = \es$.

\subsection*{Upper bounding $\mu\left(B_{2}\right)$} 

By definition of $I$, a simple union bound gives
\[
\mu\left(B_{2}\right)\le\sum_{j\in J}\frac{\mu\left(A_{x\to j}\right)}{n}
\le\frac{\left|J\right|}{n}n^{-1/3} = O(\zeta n^{-5/6}).
\]
Moreover, if $\zeta\le\frac{1}{e^{8000}\sqrt{n}}$ then $\left|J\right|<1$,
i.e.\ $J=\es$ and so $B_2 = \es$,

\subsection*{Upper bounding $\mu\left(B_{3}\right)$} 

Here we will use the bound $\mu\left(B_{3}\right)\le\mu\left(A_{x\to j}\right)$,
where we fix $j\in\overline{I\cup J}$ to maximise $\mu\left(A_{x\to j}\right)$.

It suffices to show that
$\mu\left(A_{x\to j}\right)\le\left(2e^{2000}\zeta\right)^{\frac{\sqrt{n}}{20}}$,
and moreover that if $\zeta\le\frac{1}{e^{8000}\sqrt{n}}$ then $A_{x\to j}=\es$.
We suppose for a contradiction that this fails.

We will consider $J':=J\cup\left\{ j\right\}$, noting that
$|J'| = \left\lfloor 80\cdot e^{4000}\zeta\sqrt{n}\right\rfloor + 1$
and for each $j'\in J'$ we have
$\mu\left(A_{x\to j'}\right)\ge\mu\left(A_{x\to j}\right).$

We will apply Lemma \ref{lem:large and  extremal means that the third is small}
to the product-free triple $\left(A_{i\to j'},A_{x\to i},A_{x\to j'}\right)$
for each $i\in I_{2}$ and $j'\in J'$.
By our above assumption on $C:=A_{x\to j}$
the conclusion of this lemma does not hold,
so one of the hypotheses does not hold.
By definition of $I_2$, the hypothesis for $B:=A_{x \to i}$
is satisfied with $2\zeta$ in place of $\zeta$.
Thus the hypothesis for $A_{i \to j'}$ does not hold, so
\[
\mu_{F_{I}^{x}}\left(A_{i\to j'}\right)\le1-\tfrac12 e^{-2000}.
\]
Inclusion--Exclusion then shows that
\begin{align*}
\mu_{F_{I}^{x}}\left(\overline{A}\right)
& \ge\sum_{i\in I_{2},j'\in J'} e^{-4000} \mu_{F_{I}^{x}}\left(1_{i\to j'}\right)
-\sum_{\substack{i_{1},i_{2}\in I_{2} \\ j_{1},j_{2}\in J'}}
\mu_{F_{I}^{x}}\left(1_{i_{1}\to j_{1}}\cap1_{i_{2}\to j_{2}}\right) \\
& \ge e^{-4000} |I_2||J'|/n - 2|I_2|^2|J'|^2/n^2 \ge e^{-4000} |I_2||J'|/2n.
\end{align*}
However, by assumption $\mu_{F_{I}^{x}}\left(\overline{A}\right) \le \zeta$, so
\[ |J'| \le 2e^{4000}\zeta n/|I_2| \le 4e^{4000}\zeta n/|I| < 40e^{4000}\zeta \sqrt{n},\]
which contradicts the choice of $|J'|$. Thus the required bound
for $\mu\left(A_{x\to j}\right)$ holds, so the lemma follows.
\end{proof}
We conclude this section with the proof of Theorem \ref{thm:main+},
which implies Theorem \ref{thm:main} (our main theorem).

\begin{proof}[Proof of Theorem \ref{thm:main+}]
We introduce the absolute constant $c = e^{-9000}$.
We need to show that if $n$ is sufficiently large
and $A\sub A_{n}$ is a product-free set with
$\mu(A)>\max_{I,x}\mu\left(F_{I}^{x}\right)-\frac{c}{n}$
then up to inversion $A$ is contained in some $F_{I}^{x}$.
By Theorem \ref{thm:99=000025 Stability},
there exists some $F_{I}^{x}$ such that
$\mu\left(A\setminus F_{I}^{x}\right)\le n^{-0.66}$,
where we may assume that this holds for $A$ rather than $A^{-1}$.
We note that
$\mu(F_I^x) \sim \mu(A) \sim 1/\sqrt{2en}$.
Writing $\mu_{F_{I}^{x}}\left(A\right)=1-\zeta$
we have $\zeta = O(n^{-0.16}) < e^{-20000}$.
Thus we can apply Lemma \ref{lem:bootstrapping lemma } to obtain
\[
\mu\left(A\setminus F_{I}^{x}\right)\le\zeta n^{-2/3}\log^{O\left(1\right)}n
< \zeta \mu(F_{I}^{x})/2.
\]
We deduce
\[
\mu\left(F_{I}^{x}\right)-\frac{c}{n}\le\mu\left(A\right)\le\mu\left(F_{I}^{x}\right)\mu_{F_{I}^{x}}\left(A\right)+\mu\left(A\setminus F_{I}^{x}\right)\le\mu\left(F_{I}^{x}\right)\left(1-\zeta/2\right).
\]
By choice of $c$ this implies $\zeta\le\frac{1}{e^{8000}\sqrt{n}}$,
so we can apply Lemma \ref{lem:bootstrapping lemma } to obtain $A \sub F_I^x$.
\end{proof}

\section{Concluding remarks}

Our methods for bounding product-free subsets of the alternating group
have justified the intuitive idea that globalness
can be regarded as a pseudorandomness notion,
by showing estimates on the eigenvalues of Cayley operators
for global sets that correspond to the intuition for random sets.
Given the large literature on random Cayley graphs
inspired by the seminal paper of Alon and Roichman \cite{AR},
one natural direction for further research is whether
analogous results hold for Cayley graphs with respect to global sets.

We also propose the study of extremal problems for word maps in general groups
(see the survey of Shalev \cite{Shalev} for background). Fix any group $G$.
Any word $w=w(x_1,\dots,x_d)$ in the free group $F_d$ on $d$ generators
naturally defines a word map $w:G^d \to G$. For example,
if $d=3$, $w = x_1 x_2 x_3^{-1}$, and $A,B,C \sub G$
then $A$ is product-free iff $A^3 \cap \ker w = \es$.
Our main result therefore describes the largest $A \sub A_n$
such that $A^3 \cap \ker w = \es$,
and so suggests the following more general problem.

\begin{problem}
For any finite group $G$ and word $w \in F_d$,
what is the largest $A \sub G$ with $A^d \cap \ker w = \es$?
\end{problem}

\bibliographystyle{plain}
\bibliography{ref}

\end{document}